\documentclass[reqno, 11pt]{amsart}%
\usepackage{amssymb}
\usepackage[nesting]{hyperref}
\usepackage[pdftex]{graphicx}
\usepackage{listings}
\usepackage{multirow}
\usepackage{placeins}
\usepackage{color}
\usepackage{subfigure}
\usepackage{lscape}
\usepackage{dsfont}
\usepackage{amsmath}
\usepackage{amsfonts}%
\usepackage{tikz}
\usepackage{verbatim}
\usepackage{tikz}

\providecommand{\U}[1]{\protect\rule{.1in}{.1in}}
\newcommand{\BlackBoxes}{\global\overfullrule5pt}

\BlackBoxes

\newcommand{\R}{\mathbb{R}}
\newcommand{\N}{\mathbb{N}}

\newcommand{\Eop}{\mathbb{E}}
\newcommand{\Pop}{\mathbb{P}}

\newtheorem{theorem}{Theorem}
\newtheorem{corollary}[theorem]{Corollary}
\newtheorem{lemma}[theorem]{Lemma}

\theoremstyle{definition}
\newtheorem{example}[theorem]{Example}
\newtheorem{remark}[theorem]{Remark}
\newtheorem{definition}[theorem]{Definition}
\numberwithin{equation}{section} \numberwithin{theorem}{section}
\def\0{\kern0pt\-\nobreak\hskip0pt\relax}

\makeatletter
\AtBeginDocument{ \def\@serieslogo{ \vbox to\headheight{ \parindent\z@ \fontsize{6}{7\p@}\selectfont
\vss}}}

\def\makeoverbar#1#2#3#4#5#6#7{ \setbox0=\hbox{$\m@th#2\mkern#5mu{{}#3{}}\mkern#6mu$} \setbox1=\null \dimen@=#4\fontdimen8#13 \dimen@=3.5\dimen@
\advance\dimen@ by \ht0 \dimen@=-#7\dimen@ \advance\dimen@ by \wd0
\ht1=\ht0 \dp1=\dp0 \wd1=\dimen@
\dimen@=\fontdimen8#13 \fontdimen8#13=#4\fontdimen8#13
\rlap{\hbox to \wd0{$\m@th\hss#2{\overline{\box1}}\mkern#5mu$}}
\fontdimen8#13=\dimen@}
\def\mylabel#1#2{{\def\@currentlabel{#2}\label{#1}}}
\makeatother

\begin{document}
\title[Mean Field Markov Decision Processes]{Mean Field Markov Decision Processes}
\author[N. \smash{B\"auerle}]{Nicole B\"auerle${}^*$}
\address[N. B\"auerle]{Department of Mathematics,
Karlsruhe Institute of Technology (KIT), D-76128 Karlsruhe, Germany}

\email{\href{mailto:nicole.baeuerle@kit.edu}{nicole.baeuerle@kit.edu}}

\begin{abstract}
We consider mean-field control problems in discrete time with discounted reward, infinite time horizon and compact state and action space. The existence of optimal policies is shown and the limiting mean-field problem is derived when the number of individuals tends to infinity. Moreover, we consider the average reward problem and show that the optimal policy in this mean-field limit  is $\varepsilon$-optimal for the discounted problem if the number of individuals is large and the discount factor close to one. This result is very helpful, because it turns out that  in the special case when the reward does only depend on the distribution of the individuals, we obtain a very interesting subclass of problems where an average reward optimal policy can be obtained by first computing an optimal measure from a static optimization problem and then achieving it with Markov Chain Monte Carlo methods. We give two applications: Avoiding congestion an a graph and optimal positioning on a market place which we solve explicitly.

\end{abstract}
\maketitle

\makeatletter \providecommand\@dotsep{5} \makeatother

\vspace{0.5cm}
\begin{minipage}{14cm}
{\small
\begin{description}
\item[\rm \textsc{ Key words} ]
{\small Mean-field control, Markov Decision Process, Average Reward}
\end{description}
}
\end{minipage}

\section{Introduction}
Mean-field control problems have been developed from McKean-Vlasov processes (see \cite{Mc}) where the dynamics depend on the distribution of the current state itself. In the corresponding control problem the relevant data like reward and transition function not only depend on the current state and action but also on the distribution of the state. Whereas the original motivation comes from physics these kind of problems are able to model the interaction of a large population. Thus, other popular applications include finance, queueing, energy and security problems among others. In this paper we consider mean-field control problems in discrete time in contrast to the majority of literature which concentrates on  continuous time models. Moreover, our optimization criterion is to maximize the social benefit of the system i.e. the overall  expected reward. In particular in our paper individuals cooperate in contrast to the game situation where one usually tries to find the Nash equilibrium of the system. Here we rather aim at obtaining the Pareto optimal solution. A comprehensive overview over continuous-time mean-field games can be found in \cite{cd18}. These games have been introduced in economics and later studied in mathematics since at least 15 years (see e.g. \cite{ll} for one of the first mathematical papers on this topic). 

We review briefly the latest results on {\em discrete-time} mean-field problems.
First note that there have been some early studies of interactive games in \cite{JR88} under the name {\em anonymous sequential games} and in \cite{wein05} of so-called {\em oblivious games} which are in nature very similar to mean-field games. For a recent paper on discrete-time mean-field games and a literature survey, see for example \cite{saldi}. In this paper Markov Nash equilibria are considered in a model without common noise. For an early game paper with finite state space see \cite{gomes}. Since our paper is not a game and more in the spirit of Markov Decision Processes (MDPs) we concentrate our literature survey on control papers. One of the first papers in this area have been \cite{gg11,ggb12}.  In both papers the authors' goal is to investigate the convergence of a large interacting population process to the simpler mean-field model. More precisely, the authors show convergence of value functions and convergence of optimal policies which implies the construction of asymptotically optimal policies.  In both papers the state space is finite and the action space compact. Whereas in \cite{gg11}  the convergence rate is studied, in \cite{ggb12} the authors also scale the time steps to obtain a continuous-time deterministic limit. Finite as well as infinite-horizon discounted reward problems are considered. In \cite{hjm} the authors also investigate convergence in a discounted reward problem, however consider the situation that the random disturbance density in unknown. A consumption-investment example is discussed there.  In \cite{hjm2} the same authors treat the unknown disturbance as a game against nature. The paper \cite{pw} already starts from a discrete-time mean-field control problem.  The authors derive the value iteration and solve an LQ McKean-Vlasov control problem.  In contrast to our paper there is no common noise, the authors restrict to finite time horizon and do not use MDP theory to solve their problem. However, their model data like cost and transition function may also depend on the distribution of actions. LQ-problems are popular as applications of mean-field control since it is often possible to obtain optimal policies in these cases. E.g.\ \cite{eln} is entirely devoted to these kind of problems.  

  The two papers which are closest to ours, at least as far as the model is concerned, are \cite{cltf,mp}. In both papers, the model data may also depend on the distribution of actions, but there is no restriction on admissible actions. Both consider a discounted problem with infinite time horizon. In \cite{cltf} the authors work with lower semicontinuous value functions, whereas we show continuity under the same assumptions. The main issues in \cite{cltf} are an extensive discussion of different types of policies and the development of Q-learning algorithms. We however start already with Markovian deterministic policies since in MDP theory it is well-known that history-dependent policies or randomized policies do not increase the value. Moreover, we consider the convergence of the $N$-individuals problem as well as average reward optimization.  In \cite{mp} the authors deal with so-called open-loop controls and restrict to individualized or decentralized information.  They investigate the rate of convergence from the $N$-population model to the mean-field problem. They also derive a fixed point characterization of the value function and discuss the role of randomized controls. Since in \cite{mp} decisions may only depend on the history of the single agent an additional source of randomness is required such that individuals with same history may take different actions.

Other recent papers discuss reinforcement learning for  mean-field control problems, see e.g. \cite{clt,cltf,ggw19,ggw20}. In the second part of the paper we consider average reward mean-field control problems which is a new aspect. There are papers on average reward games, like \cite{bis} where the transition probability does not depend on the empirical distribution of individuals and \cite{wi20} where under some strong ergodicity assumptions the existence of a stationary mean-field equilibrium is shown. Both papers do not consider the vanishing discount approach which we do here. The recent paper \cite{cdf} considers the vanishing discount approach, but in a continuous-time setting and for a game.  

The main contributions of our paper are as follows: We first want to stress the point that mean-field control problems fit naturally into the established MDP theory. We start with a problem where $N$ interacting individuals try to maximize their expected discounted reward over an infinite time horizon.  Reward and transition functions may depend on the empirical measure of the individuals. Moreover, the transition functions of individuals depend on an idiosyncratic noise and a common noise.  Due to symmetry reasons instead of taking the state of each individual as a common state of the system it is enough to know the empirical measure over the states. This equivalence implies an MDP formulation where the underlying state process consists of empirical measures. A similar observation can be found in \cite{mp}, however there the authors take the mean-field limit first.  Letting the number $N$ of individuals tend to infinity,  implies a mean-field limit by  applying the Glivenko-Cantelli theorem. The idiosyncratic noise vanishes in the limit. In our setting state and action spaces are compact Borel spaces. We also discuss the existence of optimal policies which is rarely done in other papers. E.g. we give explicit conditions under which an optimal {\em deterministic} policy does exist for the limit problem as well as for the initial $N$-individuals problem.  Moreover, we  investigate  average optimality in mean-field control problems, an aspect  which is neglected in the literature. Applying results from MDP theory leads to an average reward optimality inequality. In some cases we obtain optimal policies in this setting rather easily. Since we use the vanishing discount approach, we can  show  that these policies are $\varepsilon$-optimal for the initial problem when the number of individuals is large and the discount factor close to one. Thus, we get some kind of double approximation which is helpful in some applications. Indeed, it turns out that the  case when the reward does not depend on the action yields an interesting special case. The average reward problem can then be solved by first finding an optimal measure for a static optimization problem and then by using Markov Chain Monte Carlo to find an optimal randomized decision rule which achieves the optimal measure in the limit. We show how this works in a network example where the aim is to avoid congestion. Another interesting feature of the solution is that it is a decentralized control, i.e. individuals can decide optimally based on their own state without knowing the distribution of all individuals, i.e.\ individuals do not have to communicate. A second example is the optimal placement on a market square.

The paper is organized as follows: In the first section we introduce the model with a finite number of $N$ individuals. We give conditions under which the optimality equation holds and optimal policies exist. In Section \ref{sec:equiMDP} we show how to formulate an equivalent MDP whose state space consists of the empirical measures of individuals. Based on this formulation we let the number $N$ of individuals tend to infinity in the next section. We prove the convergence of value functions and show how an asymptotically optimal policy can be constructed. In Section \ref{sec:ARO} we consider the average reward problem via the vanishing discount approach. Under some ergodicity assumptions we prove the existence of average reward optimal policies and verify that the value function satisfies an average reward optimality inequality. Next we show how to use this optimal policy to construct $\varepsilon$-optimal policies for the original problem.

We discuss how to solve average reward problems when the reward depends only on the distribution of individuals  and not on the action. Finally in Section \ref{sec:app} we consider two applications (network congestion and positioning on market place) which we solve explicitly. The appendix contains additional material which consists of a useful convergence result and the definition of the Wasserstein distance and  Wasserstein ergodicity. Moreover, longer proofs are also deferred to the appendix.

\section{The Mean-Field Model}\label{sec:MFM}
 We consider the following Markov Decision Process with a finite number of individuals:
Suppose we have a compact Borel set $S$ of states and $N$ statistically equal individuals. Each individual is at the beginning in one of the  states, i.e.\ the state of the system is described by a vector $\mathbf{x}=(x_1,\ldots ,x_N)\in S^N$ which represents the states of the individuals. In case we need the time index $n$, we write $x_n^i$, $i=1,\ldots ,N$. Each individual can choose actions from the same Borel set $A$. Let $D(x)\subset A$ be the actions available for one individual who is in state $x\in S$, i.e. $\mathbf{a}=(a_1,\ldots,a_N)\in \mathbf{D}(\mathbf{x}):=D(x_1)\times\ldots\times D(x_N)$ is the vector of admissible actions for all individuals. We denote  $D:= \{ (x,a) \in S\times A : a\in D(x) \mbox{ for all } x\in S\}$ and assume that it contains the graph of a measurable mapping $f:S\to A$. Moreover, $\mathbf{D} := \{ (\mathbf{x},\mathbf{a}) | \mathbf{a}\in \mathbf{D}(\mathbf{x})\} $. After choosing an action each individual faces a random transition. In order to define this, suppose that $(Z_n^i)_{n\in\N}, i=1,\ldots ,N$ and $(Z_n^0)_{n\in \N}$ are sequences of i.i.d.\ random variables with values in a Borel set  $\mathcal{Z}$.  The sequence $(Z_n^0)_{n\in \N}$ will play the role of a common noise. In what follows we need the empirical measure of $\mathbf{x}$, i.e. we denote
$$ \mu[\mathbf{x}] := \frac1N\sum_{i=1}^N \delta_{x_i}$$
where $\delta_y$ is the Dirac measure in point $y$. $\mu[\mathbf{x}]$ can be interpreted as a distribution on $S$. We denote by $\mathbb{P}(S)$ the set of all distributions on $S$ and by 
$$ \mathbb{P}_N(S):= \{ \mu\in \mathbb{P}(S)\;| \; \mu= \mu[\mathbf{x}], \mbox{ for } \mathbf{x}  \in S^N \},$$
the set of all distributions which are empirical measures of $N$ points. On these sets we consider the topology of weak convergence. The transition function of the system is now a combination of the individual transition functions which are given by a  measurable mapping $T:
S\times A\times \mathbb{P}(S)\times \mathcal{Z}^2\to S$ such that
$$x_{n+1}^i = T(x_n^i, a_n^i, \mu[\mathbf{x}_n], Z_{n+1}^i, Z_{n+1}^0) $$
for $i=1,\ldots ,N$. Note that the individual transition may also depend on the empirical distribution $\mu[\mathbf{x}_n]$ of all individuals. In total the transition function for the entire system is a measurable mapping $\mathbf{T}: \mathbf{D} \times \mathbb{P}_N(S)\times \mathcal{Z}^{N+1}\to S^N$  of the state $\mathbf{x}$, the chosen actions $\mathbf{a}\in \mathbf{D}(\mathbf{x})$, the empirical measure $\mu[\mathbf{x}]$ and the disturbances $\mathbf{Z}_{n+1}:=(Z_{n+1}^1,\ldots, Z_{n+1}^N), Z_{n+1}^0$ such that
$$\mathbf{x}_{n+1}= \mathbf{T}(\mathbf{x}_n,\mathbf{a}_n,\mu[\mathbf{x}_n], \mathbf{Z}_{n+1}, Z_{n+1}^0)= \Big( T(x_n^i, a_n^i, \mu[\mathbf{x}_n], Z_{n+1}^i, Z_{n+1}^0)\Big)_{i=1,\ldots ,N}.$$
Last but not least each individual generates a bounded one-stage reward $r : S\times A\times \mathbb{P}(S)\to \R$ which is given by $r(x_i,a_i,\mu[\mathbf{x}])$, i.e.\ it may also depend on the empirical distribution of all individuals. The total one-stage reward of the system is the average
$$ \mathbf{r}(\mathbf{x},\mathbf{a}) :=\frac1N \sum_{i=1}^N r(x_i,a_i, \mu[\mathbf{x}])$$
of all individuals. The first aim will be to maximize the joint expected discounted reward of the system over an infinite time horizon, i.e. we consider here the {\em social optimum} of the system or {\em Pareto optimality}.  In particular the agents have to work together in order to optimize the system. This is in contrast to mean-field games where each individual tries to maximize her own expected discounted reward and where the aim is to find Nash equilibria. 
We make the following assumptions:
\begin{itemize}
\item[(A0)] $D$ is compact.
\item[(A1)] 
$x\mapsto D(x)$ is upper semicontinuous, i.e.\ for all $x\in S$: If $x_n\to x$ for $n\to\infty$ and $a_n\in D(x_n)$, then $(a_n)$ has an accumulation point in $D(x)$.
\item[(A2)]  $(x,a,\mu) \mapsto r(x,a,\mu)$ is upper semicontinuous.
\item[(A3)] $ (x,a,\mu) \mapsto T(x,a,\mu,z,z_0)$ is continuous for all $z ,z_0 \in \mathcal{Z}.$
\end{itemize}

A policy in this model is given by $\pi=(f_0,f_1,\ldots)$ with $f_n \in F$ being a decision rule where $$F:= \{ f: S^N \to A^N \;| \;  f \mbox{ is measurable } f(\mathbf{x})\in \mathbf{D}(\mathbf{x}) \mbox{ for all } \mathbf{x}\in S^N\}$$ is the set of all decision rules. In case we do not need the time index $n$ we write $f(\mathbf{x}):=(f^1(\mathbf{x}),\ldots,f^N(\mathbf{x}))$. It is not necessary to introduce randomized or history-dependent policies here, since we obtain a classical MDP below and it is well-known that an optimal policy will be among deterministic Markov ones. We assume  that each individual has information about the position of all other individuals. This point of view can be interpreted as a {\em centralized control problem } where all information is collected and shared by a central controller.

Together with the distributions of $(Z_n^i), (Z_n^0)$ and the transition function $\mathbf{T}$, a policy $\pi$ induces a probability measure $\Pop_\mathbf{x}^\pi$ on the measurable space $$(\Omega=S^N\times S^N\times \ldots, \mathcal{F}=\mathcal{B}(S^N) \otimes \mathcal{B}(S^N) \otimes \ldots)$$ where $ \mathcal{B}(S^N) $ is the Borel $\sigma$-algebra on $S^N$. The corresponding state process is denoted by $(\mathbf{X}_n)$ where $\mathbf{X}_n(\omega_1,\omega_2,\ldots)=\omega_n\in S^N$ and the action process is denoted by $(\mathbf{A}_n)$ where $\mathbf{A}_n(\omega_1,\omega_2,\ldots)=f_n(\omega_n).$ Our aim is to maximize the expected discounted reward of the system over an infinite time horizon. Hence we define for a  policy $\pi=(f_0,f_1,\ldots)$
\begin{eqnarray}
V_\pi^N(\mathbf{x})&:=&\frac1N \sum_{i=1}^N \sum_{k=0}^\infty \beta^k \Eop_\mathbf{x}^\pi\big[r(X_k^i,A_k^i, \mu[\mathbf{X}_k])\big]\\
V^N(\mathbf{x}) &:=& \sup_\pi V_\pi^N(\mathbf{x})
\end{eqnarray}
where $\beta\in (0,1)$ is a discount factor. $\Eop_\mathbf{x}^\pi$ is the expectation w.r.t.\ $\Pop_\mathbf{x}^\pi$. $V^N(\mathbf{x})$ is the maximal expected discounted reward over an infinite time horizon, initially given the configuration $\mathbf{x}$ of individual's states. 

\begin{remark}\label{rem:symmetry}
It is not difficult to see that $V^N$ is symmetric, i.e. $V^N(\mathbf{x})=V^N(\sigma(\mathbf{x}))$ for any permutation $\sigma(\mathbf{x})$ of $\mathbf{x}$ because the reward $\mathbf{r}(\mathbf{x},\mathbf{a})=\mathbf{r}(\sigma(\mathbf{x}),\sigma(\mathbf{a}))$ and the transition function $\mathbf{T}(\mathbf{x},\mathbf{a},\mu[\mathbf{x}], \mathbf{Z}, Z^0)=\mathbf{T}(\sigma(\mathbf{x}),\sigma(\mathbf{a}),\mu[\sigma(\mathbf{x})], \mathbf{Z}, Z^0)$ are symmetric. This is a simple observation but in the end leads to the conclusion that it is only necessary to know how many individuals are in the different states.
\end{remark}

In what follows we introduce some notations.

\begin{definition}
Let us define:
\begin{itemize}
\item[a)] The set $\mathbb{M} := \{v:S^N \to \R \;| \;  v \mbox{ is bounded and upper semicontinuous}\}$.
\item[b)] The operator $U$ on $\mathbb{M}$   by
\begin{eqnarray*}
Uv(\mathbf{x}) = (Uv)(\mathbf{x}) &:=& \sup_{\mathbf{a}\in \mathbf{D}(\mathbf{x})} \Big\{ \mathbf{r}(\mathbf{x},\mathbf{a})+ \beta \Eop\Big[v\big(  \mathbf{T}(\mathbf{x}, \mathbf{a}, \mu[\mathbf{x}], \mathbf{Z}, Z^0)\big)\Big]\Big\}.
\end{eqnarray*}
\item[c)] A decision rule $f\in F$ is called {\em maximizer} of $v\in \mathbb{M}$ if $$Uv(\mathbf{x})= \mathbf{ r}(\mathbf{x},f(\mathbf{x}))+ \beta \Eop\Big[v\big(  \mathbf{T}(\mathbf{x}, f(\mathbf{x}), \mu[\mathbf{x}], \mathbf{Z}, Z^0)\big)\Big].$$
\end{itemize}
\end{definition}

From classical MDP theory we obtain:

\begin{theorem}\label{thm:model1}
Assume (A0)-(A3). Then:
\begin{itemize}
\item[a)]  The value function $V^N$ is the unique fixed point of the $U$-operator in  $\mathbb{M}$, i.e. it satisfies the optimality equation $V^N=U V^N$.
\item[b)]  $V^N = \lim_{n\to\infty} U^n 0$.
\item[c)] There exists a maximizer of $V^N$ and every maximizer $f^*\in F$ of $V^N$ defines an optimal stationary (deterministic) policy $(f^*,f^*,\ldots)$.
\end{itemize}
\end{theorem}

The proof of this statement and all other longer proofs can be found in the appendix. We summarize the model data below: \\

\hrule
\vspace*{0.2cm}

\begin{tabular}{ll}
{\bf Model MDP }& \\
State space & $S^N \ni \mathbf{x}=(x_1,\ldots ,x_N)$ \\
Admissible actions &  $\mathbf{D}(\mathbf{x}):=D(x_1)\times\ldots\times D(x_N)\ni \mathbf{a}=(a_1,\ldots,a_N)$\\
Transition function & $\mathbf{T}(\mathbf{x}_n,\mathbf{a}_n,\mu[\mathbf{x}_n], \mathbf{Z}_{n+1}, Z_{n+1}^0)= \Big( T(x_n^i, a_n^i, \mu[\mathbf{x}_n], Z_{n+1}^i, Z_{n+1}^0)\Big)_{i=1,\ldots ,N}$\\
Reward & $ \mathbf{r}(\mathbf{x},\mathbf{a}) :=\frac1N \sum_{i=1}^N r(x_i,a_i, \mu[\mathbf{x}])$\\
Policy & $\pi=(f_0,f_1,\ldots),$ \\
& $f_n\in F:= \{ f: S^N \to A^N \;| \;  f \mbox{ is measurable } f(\mathbf{x})\in \mathbf{D}(\mathbf{x}),\; \forall \mathbf{x}\in S^N\}$
\end{tabular}
\hrule

\begin{example}\label{ex:running}
Suppose individuals move on a triangle. The state space is given by the nodes $S=\{1,2,3\}$. Admissible actions are adjacent nodes, i.e.\ $D(1)=\{2,3\}, D(2)=\{1,3\}, D(3)=\{1,2\}$. The individual one-stage reward may be given by $r(x_i,a_i,\mu)= 1_{\{1\}}(x_i)- 1_{\{ |1-\bar\mu|\le 0.5\}}$.
Here $\bar \mu = \int x\mu(dx)$. This means an individual gets a reward of 1 when it is in state 1, but only when the average position of the others is away from 1. A transition function may be 
$$ T(x, a, \mu, z, z^0)= \left\{ \begin{array}{cl}
a, & \mbox{ if } z\in [0,\frac12),\\
x, & \mbox{ if } z\in [\frac12,1]
\end{array}\right. $$
For $N=5$ individuals, a state may be $\mathbf{x}=(1,2,3,1,3)$ and an action $\mathbf{a}=(2,1,2,3,1)\in \mathbf{D}(\mathbf{x})$. In this case $\mu[\mathbf{x}]=(2/5,1/5,2/5)$ and $\mathbf{r}(\mathbf{x},\mathbf{a}) = 2/5$.
\end{example}

\section{The Mean-Field MDP}\label{sec:equiMDP}
Suppose that $N$ is large. Even if the state space $S$ is small, the solution of the problem may not be computationally tractable any more because $S^N$ is large. We seek for some simplifications. In particular we want to exploit the symmetry of the problem. In the last section we have seen that the empirical measures of the individuals' states is the essential information. Thus, we define as new state space $\mathbb{P}_N(S)$. Further we define the following sets:
\begin{eqnarray*}
\nonumber
\hat D (\mu) &:=& \{   \mu[(\mathbf{x},\mathbf{a})] \;| \;  \mathbf{x}\in S^N \mbox{ s.t. } \mu[\mathbf{x}] =\mu \mbox{ and } \mathbf{a}\in \mathbf{D}(\mathbf{x})\},\;  \mu\in \Pop_N(S)\\ \nonumber
\hat D &:= & \{ (\mu,Q) \;| \;  \mu \in \mathbb{P}_N(S), Q \in \hat D(\mu)\}
\\ \label{eq:GN}
\hat F &:=& \{ \varphi : \mathbb{P}_N(S) \to \mathbb{P}_N(D) \;| \;  \varphi \mbox{ measurable, }  \varphi(\mu)\in \hat D(\mu) \mbox{ for all } \mu \in  \mathbb{P}_N(S) \},
\end{eqnarray*}
where
$$\mathbb{P}_N(D) := \{ Q\in \mathbb{P}(D) \; | \; Q=  \mu[(\mathbf{x},\mathbf{a})]  \mbox{ for }  (\mathbf{x},\mathbf{a}) \in \mathbf{D} \} $$
is the set of all probability measures on $D$ which are empirical measures on $N$ points.
The set $\hat D(\mu) $ consists of probability measures on $D$ which are empirical measures on  $N$ points and whose first marginal distribution equals $\mu$.
We obtain the following result.

\begin{lemma}\label{lem:FG}
Suppose $\mathbf{a}\in \mathbf{D}(\mathbf{x}) $ is an arbitrary action in state $\mathbf{x}\in S^N$. Then there exists an admissible $Q\in \hat D (\mu[\mathbf{x}]),$   s.t.
\begin{equation}\label{eq: lem31}
\mathbf{r}(\mathbf{x},\mathbf{a})=  \int_D r(x,a,\mu)Q(d(x,a)) =:  \hat{r}(\mu,Q),
\end{equation} 
for all $\mathbf{x}\in S^N$. The converse is also true, i.e. if $Q\in \hat D (\mu[\mathbf{x}])$ then there exists an   $\mathbf{a}\in \mathbf{D}(\mathbf{x}) $  s.t.\  \eqref{eq: lem31} holds.
\end{lemma}

\begin{proof}
Let $\mathbf{x}$ and $\mathbf{a}\in \mathbf{D}(\mathbf{x})$ be given and let $\mu := \mu[\mathbf{x}]\in \Pop_N(S)$.  Define the discrete point measure $Q$ on $D$ by
$$ Q  := \mu[(\mathbf{x},\mathbf{a})]. $$  
Then $Q\in \hat D (\mu)$ by construction and
\begin{eqnarray*}
\mathbf{r}(\mathbf{x},\mathbf{a}) &=& \frac1N\sum_{i=1}^N r(x_i,a_i, \mu) =\int_{D} r(x,a,\mu)Q(d(x,a))
\end{eqnarray*}
which proves the first statement. For the converse, suppose $Q\in \hat D (\mu[\mathbf{x}])$. By definition this implies that there exists $\mathbf{a}\in \mathbf{D}(\mathbf{x})$ s.t.\ $Q=\mu[(\mathbf{x},\mathbf{a})]$. Using this relation, \eqref{eq: lem31} follows.
\end{proof}

This lemma shows that instead of choosing actions  $\mathbf{a}\in \mathbf{D}(\mathbf{x})$  we can choose measures $Q\in \hat D (\mu[\mathbf{x}])$ and $\mu=\mu[\mathbf{x}]$ is a sufficient information which can replace the high dimensional state $\mathbf{x}\in S^N$. Intuitively  this is clear from the fact that $\mathbf{r}(\mathbf{x},\mathbf{a})$ is symmetric (see Remark \ref{rem:symmetry}).

We consider now a second MDP with the following data which we will call mean-field MDP (for short $\widehat{\rm MDP}$). The state space is $ \mathbb{P}_N(S)$ and the action space is $\mathbb{P}_N(D)$. 
The one-stage reward $\hat{r}: \hat D\to \R$ is given by the expression in Lemma \ref{lem:FG}, i.e.
\begin{eqnarray}
\label{eq:hatmu}
 \hat{r}(\mu,Q) &:=& \int_{D}  r(x,a,\mu) Q(d(x,a))
 \end{eqnarray}
 and the transition law $\hat{T}: \hat D \times \mathcal{Z}^{N+1} \to \mathbb{P}_N(S)$  for $Q=\mu[(\mathbf{x},\mathbf{a})], \mu=\mu[\mathbf{x}]$ by (the empty sum is zero)
\begin{eqnarray*}
 \hat{T}(\mu,Q,\mathbf{Z},Z^0)&=& \mu[ \mathbf{T}(\mathbf{x},\mathbf{a},\mu[\mathbf{x}],\mathbf{Z},Z^0)]
 \end{eqnarray*}
The value of $\hat T$ simply is the empirical measure of the new states after a random transition.
 A policy is here denoted by $\psi=(\varphi_0,\varphi_1,\ldots)$ with $\varphi_n\in \hat F$ and we denote by $(\mu_n)$ the corresponding (random) sequence of empirical measures, i.e. $\mu_0=\mu$, and for $n\in\N_0$
 $$ \mu_{n+1} =  \hat{T}(\mu_n,\varphi_n(\mu_n),\mathbf{Z}_{n+1},Z_{n+1}^0).$$  
  
 \begin{remark}\label{rem:Q}
We define an action as a joint probability distribution $Q$ on state and action combinations instead of the conditional distribution on actions given the state. Both descriptions are equivalent, since for $Q\in \hat D(\mu)$ we can disintegrate $$Q(B)=\int_B \bar Q(da|x)\mu(dx),\; B \in \mathcal{B}(D)$$
where $\bar Q$ is the regular conditional probability. For short: $Q=\mu \otimes \bar Q$.
The advantage of using the joint distribution is that we have one object to define actions in all states. The disadvantage is that we need to formulate the restriction that the marginal distribution on the states coincides with $\mu$.
\end{remark}
 
 We define  the value function of $\widehat{\rm MDP}$  in the usual way for state $\mu\in \mathbb{P}_N(S)$ and policy $\psi=(\varphi_0,\varphi_1,\ldots)$ by
\begin{eqnarray}
J_\psi^N(\mu) &:=& \sum_{k=0}^\infty \beta^k \Eop_\mu^\psi\big[\hat{r}(\mu_k,\varphi_k)\big].\\
J^N(\mu)&:=&\sup_\psi J^N_\psi(\mu).
\end{eqnarray}

Finally, we show that the MDP and the mean-field MDP are equivalent. 

\begin{theorem}\label{thm:equality}
Assume (A0)-(A3). For $\mathbf{x}\in S^N$ and $\mu=\mu[\mathbf{x}]$ we have:
$$ V^N(\mathbf{x})=J^N(\mu).$$
\end{theorem}

\begin{proof}
Note that $\mu_0=\mu=\mu[\mathbf{x}]$ by definition.
Let  $\mathbf{a}_0=\mathbf{a}\in \mathbf{D}(\mathbf{x})$ be the first action taken by MDP under an arbitrary policy. Then by  Lemma \ref{lem:FG} there exists $Q\in \hat D(\mu)$, s.t. $\mathbf{r}(\mathbf{x},\mathbf{a})=   \hat{r}(\mu,Q)$ and 
$$\mu[\mathbf{X}_{1}]=\mu[\mathbf{T}(\mathbf{x},\mathbf{a},\mu[\mathbf{x}], \mathbf{Z}_{1}, Z^0_{1})]=\hat{T}(\mu,Q, \mathbf{Z}_{1},Z^0_{1})=\mu_{1}.$$
By induction over time $n$ it follows that a sequence of states and feasible actions in MDP $(\mathbf{X}_0,\mathbf{A}_0,\mathbf{X}_1,\ldots)$ can be coupled with a sequence of states and feasible actions  $(\mu_0,Q_0,\mu_1,\ldots)$ for $\widehat{\rm MDP}$ and vice versa  s.t.\ the same sequence of disturbances $(\mathbf{Z}_n),(Z^0_n)$ is used and $  \mathbf{r}(\mathbf{X}_n,\mathbf{A}_n) = \hat{r}(\mu_n,Q_n)$
pathwise. The corresponding policies may be history-dependent, but $V^N=J^N$ follows since it is well-known for MDPs that the maximal value is obtained when we restrict our optimization to Markovian policies.
\end{proof}

As in Section \ref{sec:MFM} we define here a set and an operator for the mean-field MDP.

\begin{definition}
Let us define
\begin{itemize}
\item[a)] The set $\mathbb{\hat M}:= \{v: \mathbb{P}_N(S) \to \R \; |\; v \mbox{ is bounded and upper semicontinuous}\}$.
\item[b)] The operator $\hat U$ on $\mathbb{\hat M}$  by
\begin{eqnarray*}
\hat Uv(\mu) = (\hat Uv)(\mu) &:=& \sup_{Q\in \hat D(\mu)} \Big\{ \hat{r}(\mu,Q)+ \beta \Eop v( \hat{ T}(\mu,Q,\mathbf{Z},Z^0))\Big\}.
\end{eqnarray*}
\end{itemize}
\end{definition}

Due to Theorem \ref{thm:equality} and Theorem \ref{thm:model1} we obtain:

\begin{theorem}\label{theo:JN}
Assume (A0)-(A3). Then:
\begin{itemize}
\item[a)] The value function $J^N$ is the unique fixed point of the $\hat U$-operator in $\mathbb{\hat M}$ i.e. it satisfies the optimality equation $J^N= \hat U J^N$.
\item[b)]   $J^N = \lim_{n\to\infty } \hat U^n0$.
\item[c)] There exists a maximizer of $J^N$ and every maximizer $\varphi^*\in \hat F$ of $J^N$ defines an optimal stationary  policy $(\varphi^*,\varphi^*,\ldots)$.
\end{itemize}
\end{theorem}


We summarize the model data below: \\

\hrule
\vspace*{0.2cm}

\begin{tabular}{ll}
{\bf Model $\widehat{\rm MDP}$ }& \\
State space & $\mathbb{P}_N(S):= \{ \mu\in \mathbb{P}(S)\;| \; \mu= \mu[\mathbf{x}], \mbox{ for } \mathbf{x}  \in S^N \} \ni \mu$ \\
Admissible actions &  $\hat D (\mu) :=\{   \mu[(\mathbf{x},\mathbf{a})] \;| \;  \mathbf{x}\in S^N \mbox{ s.t. } \mu[\mathbf{x}] =\mu \mbox{ and } \mathbf{a}\in D(\mathbf{x})\} \ni Q$\\
Transition function & $ \hat{T}(\mu,Q,\mathbf{Z},Z^0)= \mu[ \mathbf{T}(\mathbf{x},\mathbf{a},\mu[\mathbf{x}],\mathbf{Z},Z^0)]$\\
Reward & $  \hat{r}(\mu,Q) := \int_{D}  r(x,a,\mu) Q(d(x,a))$\\
Policy & $\psi=(\varphi_0,\varphi_1,\ldots),$ \\
& $\varphi_n\!\in\! \hat F\! :=\! \{ \varphi : \mathbb{P}_N(S) \to \mathbb{P}_N(D) \;| \;  \varphi \mbox{ meas., }  \varphi(\mu)\in \hat D(\mu),\; \forall \mu \in  \mathbb{P}_N(S) \} $
\end{tabular}
\hrule
\begin{example}\label{ex:running2}
We reconsider Example \ref{ex:running}. The given state and action translates in $\widehat{\rm MDP}$  to $\mu=\mu[\mathbf{x}]=(2/5,1/5,2/5)$ as distribution on $S=\{1,2,3\}$. The action is a distribution on $D=\{(1,2),(1,3),(2,1),(2,3),(3,1),(3,2)\}$ and translates into $Q=(1/5,1/5,1/5,0,1/5,1/5)$. The transition kernel mentioned in Remark \ref{rem:Q} in this example is given by $\bar Q(2|1)=\frac12,\bar Q(3|1)=\frac12,\bar Q(1|2)=1, \bar Q(3|2)=0, \bar Q(1|3)=\frac12, \bar Q(2|3)=\frac12.$  Obviously $ \hat{r}(\mu,Q) =2/5$.
\end{example}

\section{The Mean-Field Limit MDP}\label{sec:MFL}
In this section we let $N\to\infty$ in order to obtain some simplifications. This yields the so-called mean-field limit. 

We thus consider a third MDP, the  so-called limit MDP (denoted by $\widetilde{\rm MDP}$). We will later show that it will indeed appear to be the limit of the problems studied in the previous section. The limit MDP is defined by  the following data: The state space is $ \mathbb{P}(S)$ and the action space is $\mathbb{P}(D)$.  We define
\begin{eqnarray}
\tilde D(\mu) &:=& \{ Q \in \mathbb{P}(D) \:| \mbox{ the first margin of } Q \mbox{ is } \mu \},\, \mu\in \Pop(S)\\
\tilde D &:=& \{(\mu,Q)\: |\: \mu \in \mathbb{P}(S), Q\in \tilde D(\mu)\}.
\end{eqnarray}

The one-stage reward $\tilde{r}: \tilde D \to \R$ is given as in
\eqref{eq:hatmu}: 
$$ \tilde{r}(\mu,Q) := \int_{D}  r(x,a,\mu) Q(d(x,a)).$$
The transition function is defined by $\tilde{T}: \tilde D \times \mathcal{Z}\to \mathbb{P}(S)$ 
\begin{equation}\label{eq:tile_T}
\tilde{T}(\mu,Q,Z^0)(B) = \int_{ D}  p^{x,a,\mu,Z^0}(B)  Q(d(x,a))
\end{equation} 
where $p^{x,a,\mu,Z^0} (B):=\Pop(T(x,a,\mu,Z^i,Z^0)\in B |Z^0)$ with $B\in \mathcal{B}(S)$, is the conditional probability that the next state is in $B$, given $x,a,\mu$ and the common noise random variable $Z^0$.  

\begin{remark}
Recalling that $Q\in \tilde D(\mu)$ means  $Q=\mu \otimes \bar{Q}$, we can (with the help of the Fubini theorem)  instead of \eqref{eq:tile_T} equivalently write
\begin{eqnarray}
\tilde{T}(\mu,Q,Z^0)(B)  &=&  \int_D p^{x,a,\mu,Z^0}(B)  \bar{Q}(da|x) \mu(dx)\\
&=&  \int_S P^{\bar Q,\mu,Z^0}(B|x) \mu(dx) 
\end{eqnarray} 
where $P^{\bar Q,\mu,Z^0} (dx'|x)= \int_{D(x)} p^{x,a,\mu,Z^0}(dx') \bar{Q}(da|x)$. Hence $P^{\bar Q,\mu,Z^0}$ is the transition kernel which determines the distribution at the next stage. In general it depends on $\bar Q,\mu$ and the common noise $Z^0$. 
\end{remark} 

A decision rule is here  a measurable mapping $\varphi$ from $\mathbb{P}(S)$ to $\mathbb{P}(D)$ such that $\varphi(\mu)\in \tilde D(\mu)$ for all $\mu$. We denote by $\tilde F$ the set of all decision rules. Suppose that  $\psi=(\varphi_0,\varphi_1,\ldots)$ is a policy for the $\widetilde{\rm MDP}$. As in the previous section we set for $n\in \N_0$
\begin{eqnarray*}
\mu_0 &:=& \mu,\\
\mu_{n+1} &:=& \tilde T(\mu_n,\varphi_n(\mu_n), Z^0_{n+1})
\end{eqnarray*}
which yields the sequence of distributions of individuals. Note that it is deterministic if $\tilde T$ does not depend on the common noise $Z^0$.

Then we define for $\widetilde{\rm MDP}$ the following value functions for policy $\psi=(\varphi_0,\varphi_1,\ldots)$ and state $\mu\in \mathbb{P}(S)$
\begin{eqnarray}
\nonumber
 J_{\psi}(\mu) &=& \sum_{k=0}^\infty \beta^k \Eop_\mu^\psi[\tilde r(\mu_k, \varphi_k)], \\ \label{eq:mfl}
 J(\mu) &=& \sup_\psi  J_\psi(\mu).
\end{eqnarray}

Instead of (A2) we will now assume that

\begin{itemize}
\item[(A2')]  $ (x,a,\mu) \mapsto r(x,a,\mu)$ is continuous.
\end{itemize}

\begin{definition}
We define
\begin{itemize}
\item[a)] The set 
$\tilde{\mathbb{M}} := \{ v: \mathbb{P}(S) \to \R \; |\; v \mbox{ is continuous and bounded}\}$.
\item[b)] The maximal reward operator $\tilde U$ on $\tilde{\mathbb{M}} $  in this model   is 
\begin{eqnarray*}
\tilde U v(\mu)= (\tilde U v)(\mu) &:=& \sup_{Q\in \tilde D (\mu) } \Big\{ \tilde{r}(\mu,Q)+ \beta \Eop v(  \tilde{T}(\mu,Q,Z^0))\Big\}.
\end{eqnarray*}
\end{itemize}
\end{definition}

For the mean-field limit MDP we obtain:
 
\begin{theorem}\label{theo:MFLexi}
Assume (A0), (A1), (A2'), (A3). Then: 
\begin{itemize}
\item[a)] The  value function $J$ is the unique fixed point of the $\tilde U$-operator in $\tilde{\mathbb{M}}$, i.e.\ it satisfies the optimality equation $J= \tilde U J$.
\item[b)]  $J = \lim_{n\to\infty} \tilde U^n 0$. 
\item[c)] There exists a maximizer of $J$ and every maximizer $\varphi^*\in \tilde F$ of $J$ defines an optimal stationary deterministic policy $(\varphi^*,\varphi^*,\ldots)$.
\end{itemize}
\end{theorem}

\begin{remark}
We can use the established solution methods like value iteration, policy iteration,  linear programmes or reinforcement learning to numerically solve the limit MDP (\cite{BT,CHFM,P07}). 
\end{remark}

The limit problem can be seen as a problem which approximates the original model when $N$ is large. In order to proceed, we need a more restrictive assumption than (A3)

\begin{itemize}
\item[(A3')] $\mathcal{Z}$ is compact and $ (x,a,\mu,z,z_0) \mapsto T(x,a,\mu,z,z_0)$ is continuous.
\end{itemize}

\begin{remark}
The assumption that $\mathcal{Z}$ is compact is not a strong assumption. Indeed, w.l.o.g.\ we may choose the disturbances to be uniformly distributed over $[0,1].$ This is because if  for example $\mathcal{Z}=\R$ and  $F$ is the distribution function of $Z$ we get $Z\stackrel{d}{=} F^{-1}(U)$ with $U\sim U([0,1])$ and $F^{-1}$ is then part of the transition function.
\end{remark}

Then it is possible to prove the following limit result.  

\begin{theorem}\label{thm:limit1}
Assume (A0), (A1), (A2') and  (A3'). Let $\mu^N_0\Rightarrow \mu_0$ for $N\to\infty$  where $\mu_0^N \in \mathbb{P}_N(S)$.  Then 
\begin{itemize}
\item[a)] $\limsup_{N\to\infty} J^N(\mu^N_0)= J(\mu_0)$.
\item[b)] Suppose $\varphi^*$ is a maximizer  of $J$. Then  it is possible to construct  (possibly history-dependent) policies $\psi^N= (\varphi^N_0,\varphi^N_1,\ldots)$  for $\widehat{\rm MDP}$ s.t.\ $\lim_{N\to\infty} J^N_{\psi^N}(\mu_0^N)=J(\mu_0)$.
\end{itemize} 
\end{theorem}

In particular the proof of  part b) shows how to obtain an $\varepsilon$-optimal policy for the model with $N$ individuals ($N$ large) when we know the optimal policy for the limit MDP.

\begin{remark}\label{rem:randomized}
\begin{itemize}
\item[a)] In case there is no common noise,  $\widetilde{\rm MDP}$ is completely deterministic.  The optimality equation then reads
\begin{eqnarray}\label{eq:JQT}
J(\mu) &=& \sup_{Q\in \tilde D(\mu)} \Big\{ \tilde r(\mu,Q)+\beta J(\tilde T(\mu,Q))\Big\} 
\end{eqnarray} 
where $\tilde T(\mu,Q)(B) = \int p^{x,a,\mu}(B) Q(d(x,a))$ with $p^{x,a,\mu}(B) = \Pop(T(x,a,\mu,Z)\in B)$.

\item[b)] If there is no common noise and $r$ and $T$ do not depend on $\mu$,  we obtain as a special case a standard MDP. The usual optimality equation for this MDP (for one individual)  would be
\begin{equation}\label{eq:vistandard}
 V(x) = \sup_{a\in D(x)} \left\{ r(x,a)+ \beta \Eop V(T(x,a,Z))\right\},\; x\in S
\end{equation}
where $V(x) = \sup_\pi \sum_{k=0}^\infty \beta^k \Eop_x^\pi[r(X_k^i,A_k^i)]$.
The results in this paper  show that 
 we can equivalently consider $\widehat{\rm MDP}$ which implies the optimality equation \eqref{eq:JQT}.  It is possible to show by induction that the relation between both value functions is given by
$J(\mu) =  \int V(x)\mu(dx)$. Moreover, a maximizer of $J$ is given by  $\varphi^*(\mu)=\mu \otimes \bar Q^*$ with $\bar Q^*(\cdot|x)= \delta_{f^*(x)}$ for some $f^*: S\to A$ with $f^*(x)\in D(x)$ and $f^*$ is a  maximizer of $V$.
Here the choice of the conditional distribution $\bar Q^*$ does not depend on $\mu$ and is concentrated on a single action.
\item[c)] The policy $\psi^N$ which is constructed in Theorem \ref{thm:limit1} is {\em deterministic} but has the disadvantage that individuals have to communicate. Another possibility is to choose $Q_0^N$ as an empirical measure of $Q_0^*$ given $\mu_0^N$. This means if $Q_0^* = \mu_0 \otimes \bar Q^*$ and $\mu_0^N = \mu[\mathbf{x}^N]$ then simulate for all $x_i^N$ actions $a_i^N$ according to the kernel $\bar Q^*$. This is then a {\em randomized} policy but has the advantage that every individual can do this on its own without having the information about the other states and actions. This is then a decentralized control, i.e. $f^i(\mathbf{x})=f^i(x_i)$. Also the speed of the convergence in Theorem \ref{thm:limit1}  depends on the chosen approximation method.
\end{itemize}
\end{remark}

We summarize the model data below: \\

\hrule
\vspace*{0.2cm}

\begin{tabular}{ll}
{\bf Model $\widetilde{\rm MDP}$ }& \\
State space & $\mathbb{P}(S) \ni \mu$ \\
Admissible actions &  $\tilde D (\mu) :=\{ Q \in \mathbb{P}(D) \:| \mbox{ the first margin of } Q \mbox{ is } \mu \}\ni Q$\\
Transition function & $ \tilde{T}(\mu,Q,Z^0)(B) = \int_{ D}  p^{x,a,\mu,Z^0}(B)  Q(d(x,a))$ where\\
&  $p^{x,a,\mu,Z^0} (B):=\Pop(T(x,a,\mu,Z^i,Z^0)\in B |Z^0)$ \\
Reward & $   \tilde{r}(\mu,Q) := \int_{D}  r(x,a,\mu) Q(d(x,a))$\\
Policy & $\psi=(\varphi_0,\varphi_1,\ldots),$ \\
& $\varphi_n\in \tilde F := \{ \varphi : \mathbb{P}(S) \to \mathbb{P}(D) \;| \;  \varphi \mbox{ meas., }  \varphi(\mu)\in \tilde D(\mu),\; \forall \mu \in  \mathbb{P}(S) \} $
\end{tabular}
\hrule

\begin{example}\label{ex:running2}
We reconsider Example \ref{ex:running}. In $\widetilde{\rm MDP}$ a state can be any distribution on $S$, e.g. $\mu=(\pi^{-1},0,1-\pi^{-1})$. An  action is a distribution on $D=\{(1,2),(1,3),(2,1),(2,3),(3,1),(3,2)\}$ s.t.\ the first margin is $\mu$. For example $Q=(\pi^{-1},0,0,0,3/4(1-\pi^{-1}),1/4(1-\pi^{-1}))$.  Here $ \tilde{r}(\mu,Q)  =\pi^{-1}$.
\end{example}

\section{Average Reward Optimality}\label{sec:ARO}
In this section we consider the problem of finding the maximal average reward of the mean-field limit problem $\widetilde{\rm MDP}$. So suppose an $\widetilde{\rm MDP}$ as in the previous section (equation \eqref{eq:mfl}) is given. For a fixed policy $\psi=(\varphi_1,\varphi_2,\ldots)$ define
\begin{equation}\label{eq:AR}
\liminf_{n\to\infty} \frac1n \sum_{k=0}^{n-1} \Eop_\mu^\psi [\tilde r (\mu_k,\varphi_k)] =: G_\psi(\mu).
\end{equation}
The problem is to find $G(\mu) := \sup_\psi G_\psi(\mu)$ for all $\mu\in \Pop(S)$. We will construct the solution via the {\em vanishing discount approach}, see e.g. \cite{sen09,schal93,HLL,b01}. This has the advantage that we get a statement about the approximation of the $\beta$-discounted problem by the average reward problem immediately. For this purpose we denote by $J^\beta, J^\beta_ \psi$ the value functions of the discounted reward problem $\widetilde{\rm MDP}$ of the previous section in order to stress that they depend on the discount factor $\beta$.  

We first note that the following Tauber Theorem holds (see e.g. \cite{sen09}, Th. A.4.2):

\begin{lemma}\label{lem:Tauber}
For arbitrary $\mu\in \mathbb{P}(S)$ and policy  $\psi=(\varphi_0,\varphi_1,\ldots)$ we have
\begin{eqnarray*}
&&\liminf_{n\to\infty} \frac1n \sum_{k=0}^{n-1} \Eop_\mu^\psi [\tilde r (\mu_k,\varphi_k)]  = G_\psi(\mu) \le \liminf_{\beta\uparrow 1} (1-\beta) J_\psi^\beta(\mu)\\
&\le & \limsup_{\beta\uparrow 1} (1-\beta) J_\psi^\beta(\mu)\le
\limsup_{n\to\infty} \frac1n \sum_{k=0}^{n-1} \Eop_\mu^\psi [\tilde r (\mu_k,\varphi_k)]  <\infty
\end{eqnarray*}
\end{lemma}

In order to proceed we make the following assumption (compare with condition (B) in \cite{schal93} or condition (SEN) in \cite{sen09}, Section 7.2). 

\begin{itemize}
\item[(A4)]  There exist $L>0, \bar \beta \in (0,1)$ and a function $M: \Pop(S)\to\R$ such that
$$ M(\mu) \le h^\beta(\mu) := J^\beta(\mu)-J^\beta(\nu)\le L$$
for fixed $\nu\in \Pop(S)$, all $\mu\in \Pop(S)$ and all $\beta \ge \bar \beta$.
\end{itemize}

We define $\rho(\beta) := (1-\beta) J^\beta(\nu)$. Note that since $r$ is bounded by a constant $C>0$ say, we obtain $ |\rho(\beta)| \le  (1-\beta) |J^\beta(\nu)| \le C$. I.e. $ \rho(\beta)$ is bounded and $\limsup_{\beta\uparrow 1} \rho(\beta)=:\rho$ exists. Now we obtain:

\begin{lemma}\label{lem:UB}
Under (A4) there exists a sequence $(\beta_n)$ with $\lim_{n\to\infty} \beta_n = 1$ s.t.\ $$ \lim_{n\to\infty} (1-\beta_n) J^{\beta_n}(\mu)=\rho$$
for all $\mu\in \Pop(S)$. In particular we have $G_\psi(\mu) \le \rho$ for all $\mu$ and $\psi$.
\end{lemma} 

\begin{proof}
Using (A4) we obtain:
\begin{eqnarray*}
|(1-\beta) J^\beta(\mu)-\rho| &=& |(1-\beta) h^\beta(\mu) + \rho(\beta)-\rho | \le 
 (1-\beta) |h^\beta(\mu)| + |\rho(\beta)-\rho | \\
 &\le &  (1-\beta) \max\{L,M(\mu)\} + |\rho(\beta)-\rho |.
\end{eqnarray*}
The last term converges to zero when we choose $(\beta_n)$ s.t.\ $\lim_{n\to\infty}\beta_n=1$ and $\lim_{n\to\infty}\rho(\beta_n)=\rho$ which is possible due to the considerations preceding this lemma. The first term also tends to zero.
\end{proof}

We obtain:

\begin{theorem}\label{thm:average}
Assume (A0), (A1), (A2'), (A3'), (A4). Then:
\begin{itemize}
\item[a)] There exists a constant $\rho\in\R$ and an upper semicontinuous function $h :\Pop(S)\to\R$ such that the average reward optimality inequality holds, i.e. for all $\mu\in\Pop(S)$
\begin{equation}\label{eq.AROI}
\rho + h(\mu) \le \sup_{Q\in \tilde D(\mu)} \left\{ \tilde r(\mu, Q) + \Eop[ h(\tilde T(\mu,Q,Z^0))]  \right\}.
\end{equation}
Moreover, there exists a maximizer $\varphi^*$ of \eqref{eq.AROI}. 
\item[b)] The stationary policy $(\varphi^*,\varphi^*,\ldots)$ is optimal for the average reward problem and $\rho = \limsup_{\beta\uparrow 1} \rho(\beta)$ is the maximal average reward, independent of $\mu$. Moreover, there exists a decision rule $\varphi^0$ and sequences $\beta_m(\mu)\uparrow 1$, $\mu_m(\mu) \to \mu$ s.t.
$$ \varphi^0(\mu) := \lim_{m\to\infty} \varphi^{\beta_m(\mu)}(\mu_m(\mu))$$
where $\varphi^\beta$ is an optimal decision rule in the $\beta$-discounted model and the stationary policy $(\varphi^0,\varphi^0,\ldots)$ is optimal for the average reward problem.
\end{itemize}
\end{theorem}

Note that part b) of the previous theorem states that it is possible to obtain an  average reward optimal policy from optimal policies in the discounted model. Indeed what is maybe more interesting is the converse. From the average optimal policy we can construct $\varepsilon$-optimal policies for $\widetilde{\rm MDP}$ and thus also for $\widehat{\rm MDP}$ if $\beta$ is close to one.   The idea is to use the double approximation (number of agents large, discount factor large) to
approximate the discounted finite agent model by the average mean-field problem. We do not tackle the question of convergence speed or how $\beta$ depends on $N$  here.  A policy $\psi$ is $\varepsilon$-optimal in state $\mu\in \mathbb{P}(S)$ for $\widetilde{\rm MDP}$ if 
$$ 1- \Big| \frac{J_\psi^\beta(\mu)}{J^\beta(\mu)}\Big| \le \varepsilon. $$ 
Thus, we obtain:

\begin{corollary}
Under the assumptions of Theorem \ref{thm:average} suppose $\psi^*=(\varphi^*,\varphi^*,\ldots)$ is an optimal stationary policy for the average reward problem and $\psi^N$ is constructed as in Theorem \ref{thm:limit1}. Then for all $\varepsilon>0$ and for all $\mu\in \mathbb{P}(S)$ there exists a $\beta(\mu) <1$ 
\begin{itemize}
\item[a)]   s.t.\ $\psi^*$ is $\varepsilon$-optimal for $\widetilde{\rm MDP}$ in state $\mu$ for all $\beta\ge \beta(\mu)$.
\item[b)]  and there exists a $N(\mu,\beta(\mu))\in\N$ s.t.\ for all $N\ge N(\mu,\beta(\mu))$ and $\beta \ge \beta(\mu)$  $\psi^N$ is $\varepsilon$-optimal for $\widehat{\rm MDP}$, i.e. $(1-\beta)|J^N_{\psi^N}(\mu^N)-J^N(\mu^N) |\le \varepsilon$ where $\mu^N \Rightarrow \mu$.
\end{itemize}
\end{corollary}

\begin{proof}
\begin{itemize}
\item[a)] By Theorem \ref{thm:average} we know that $\rho=G_{\psi^*}(\mu)$ is the maximal average reward. Lemma \ref{lem:Tauber} and Theorem \ref{thm:average}  together imply 
\begin{eqnarray*}
\rho&=&G_{\psi^*}(\mu)\le \liminf_{\beta\uparrow 1}(1-\beta) J^\beta_{\psi^*}(\mu) \le \limsup_{\beta\uparrow 1}(1-\beta) J^\beta_{\psi^*}(\mu) \\
&\le& \limsup_{\beta\uparrow 1}(1-\beta) J^\beta(\mu) =\rho 
\end{eqnarray*}
which means that we have equality everywhere. Since $r$ is bounded, w.l.o.g.\ we may assume that $r$ is bounded from below by $\underline C>0$, otherwise we have to shift the function by a constant. Now for all $\varepsilon>0$ we can choose, due to the preceding equation, $\beta(\mu)$ s.t.\ for all $\beta\ge \beta(\mu)$
$$ |J^\beta(\mu)-J^\beta_{\psi^*}(\mu)|\le \frac{\varepsilon}{1-\beta}
\mbox{   and hence   } 1- \Big| \frac{J_{\psi^*}^\beta(\mu)}{J^\beta(\mu)}\Big| \le \frac{\varepsilon}{(1-\beta) J^\beta(\mu)}\le \frac{\varepsilon}{\underline C}$$
which implies the result.
\item[b)] Let  $\varepsilon>0$. From part a) choose $\beta(\mu)<1$ s.t. for all $\beta\ge \beta(\mu)$ we have $(1-\beta) |J^\beta(\mu)-J^\beta_{\psi^*}(\mu)|\le\varepsilon/3.$  Fix such a $\beta\ge \beta(\mu)$. From Theorem  \ref{thm:limit1}  choose $N\ge N(\mu,\beta)$  s.t.\
$$ |J^N_{\psi^N}(\mu^N)-J^\beta_{\psi^*}(\mu) |\le \varepsilon/3 \mbox{  and  }
 |J^N(\mu^N)-J^\beta(\mu) |\le \varepsilon/3.$$
 Then, in total
 \begin{eqnarray}\nonumber
 (1-\beta)|J^N_{\psi^N}(\mu^N)-J^N(\mu^N)| &\le & (1-\beta)|J^N_{\psi^N}(\mu^N)-J^\beta_{\psi^*}(\mu)| + (1-\beta)|J^\beta_{\psi^*}(\mu)-J^\beta(\mu)|\\
 && + (1-\beta)|J^\beta(\mu)-J^N(\mu^N)| \le \varepsilon
 \end{eqnarray}
 which implies the statement.
\end{itemize}
\end{proof}


\subsection{Special Case I}
We consider the following special case: The reward depends only on $\mu$, i.e. we have $\tilde r(\mu,Q)=\tilde r (\mu)$. The transition function is independent of $\mu$ and there is no common noise, i.e. all individuals move independently from each other.  Suppose $\mu^*\in \Pop(S)$ is the solution of the static optimization problem
\begin{equation} \label{eq:optimizprobmudmu}
\left\{ \begin{array}{ll}
\max \tilde r(\mu) \\
s.t. \; \mu \in \mathbb{P}(S)
\end{array}\right.
\end{equation}
which exists since $r$ is continuous on the compact space $\mathbb{P}(S)$. In the described situation $\widetilde{\rm MDP}$ is deterministic and the evolution of the state process for a given policy is
\begin{equation}\label{eq:stateevolution}
\mu_{k+1}(B) = \int_D p^{x,a}(B) \bar Q (da|x)\mu_k(dx) = \int_S P^{\bar Q}(B|x)\mu_k(dx) ,\; B\in \mathcal{B}(S)
\end{equation} 
for $k\in\N$ where we start with the initial distribution $\mu_0$.

Now suppose further that there exists a transition kernel (policy) $ \bar Q^*$ such that $\mu^*$ is a stationary distribution of $ P^{ \bar Q^*}$ and  $P^{ \bar Q^*}$ satisfies the Wasserstein ergodicity (see Appendix). Suppose further that $(\mu_k^*)$ is the state sequence obtained in \eqref{eq:stateevolution} where we replace $P^{\bar Q}$ by $P^{ \bar Q^*}$. Then $\mu_k^*\Rightarrow\mu^*$ for $k\to\infty$ weakly since convergence in the Wasserstein metric implies weak convergence on compact sets. Problem \eqref{eq:optimizprobmudmu} and the solution approach here is similar to the concept of {\em steady state policies} in \cite{F79}.

\begin{lemma}
Under the assumptions of this subsection  $\varphi^*(\mu) = \mu \otimes \bar Q^*$ defines an average reward optimal stationary policy $\psi^*=(\varphi^*,\varphi^*,\ldots)$.
\end{lemma}

\begin{proof}
Since $\mu\mapsto \tilde r(\mu)$ is continuous  (see proof of Theorem   \ref{theo:MFLexi}) we obtain  $\lim_{k\to\infty}\tilde r(\mu_k^*) \to  \tilde r(\mu^*)$.
Thus we have for all $\mu\in \Pop(S)$
\begin{equation*}\label{eq:AR2}
G_{\psi^*}(\mu) =\liminf_{n\to\infty} \frac1n \sum_{k=0}^{n-1}  \tilde r (\mu_k^*) = \tilde r(\mu^*)=  G(\mu).
\end{equation*}
The last equation follows from the definition of $\mu^*$. Hence $\psi^*$ is average reward optimal.
\end{proof}

 We can think of the problem  thus been transformed into a Markov Chain Monte Carlo problem to sample from $\mu^*$. In order to obtain an $\varepsilon$-optimal policy in the $N$ individual problem with large discount factor, an individual in state $x$ can sample its action from $\bar Q^*(\cdot |x)$ (see proof of Theorem \ref{thm:limit1} and Remark \ref{rem:randomized} c)). This yields a decentralized decision which does not depend on the complete state of the system. I.e. the individuals do not have to communicate with each other in order to push the system to the social optimum. The knowledge about the own state is sufficient. Problems may occur when the solution of \eqref{eq:optimizprobmudmu} is not unique. Then the individuals have to communicate which solution is preferred.  In particular the individual's optimal decision coincides with the social optimal decision. This is because we can interpret $\mu_k$ as the distribution of a typical individual at time $k$. Also note that in this case it can be shown that Assumption (A4) is satisfied since $|\tilde r(\mu_k^*)-\tilde r(\mu^*)| \le C W(\mu_k^*,\mu^*) \le \tilde C \rho^k$ with $\rho\in(0,1)$ where $W$ is the Wasserstein distance of two measures (see Appendix). We will give a more specific application in Section \ref{sec:app}.

\subsection{Special Case II}
We relax the previous case and allow the transition function to depend on $\mu$.
 Again we determine the solution $\mu^*$ of \eqref{eq:optimizprobmudmu}  first. Next we check whether there exists a transition kernel (policy) $\bar Q^*$ such that $\mu^*$ is a stationary distribution of $ P^{\bar Q^*}$ with $P^{\bar Q^*}(B|x)= \int p^{x,a,\mu^*}(B) \bar Q^* (da|x)$ for $x\in S, B\in \mathcal{B}(S)$ and  $P^{\bar Q^*}$ satisfies the Wasserstein ergodicity. Here, we need some further properties of the model to obtain the same result as in Case I, because we have to make sure that the system still converges to $\mu^*$, even if we choose the 'wrong' transition kernel
$$ \int p^{x,a,\mu_k}(B) \bar Q^* (da|x)$$
at stage $k$. Note that  the evolution of the state in this model is given by
$$ \mu_{k+1}^*(B) = \int\int p^{x,a,\mu_k}(B) \bar Q^* (da|x)\mu_k^*(dx).$$
 In particular we want to find an optimal decentralized control. The following assumptions will be useful:
\begin{itemize}
\item[(T1)]  There exists $\gamma_W>0$  s.t.\ $\sup_{x,a,z}|T(x,a,\mu,z)-T(x,a,\mu^*,z)|\le \gamma_W W(\mu,\mu^*)$ for all $\mu\in \Pop(S)$.
\item[(T2)] $D(x)$ does not depend on $x$ and $W(\bar Q^*(\cdot |x), \bar Q^*(\cdot |x'))\le \gamma_Q |x-x'|$ for all $x,x'\in S$.
\item[(T3)]  There exists $\gamma_A>0$  s.t.\ $\sup_{x,z}|T(x,a,\mu^*,z)-T(x,a',\mu^*,z)|\le \gamma_A |a-a'|$ for all $a,a'\in A$.
\item[(T4)]  There exists $\gamma_S>0$  s.t.\ $\sup_{a,z}|T(x,a,\mu^*,z)-T(x',a,\mu^*,z)|\le \gamma_S |x-x'|$ for all $x,x'\in S$.
\item[(T5)] $\gamma:=\gamma_W+\gamma_Q\gamma_A+\gamma_S<1.$  
\end{itemize}
The next lemma  states that  under these assumptions the sequence $(\mu_k^*)$ still converges against the optimal distribution $\mu^*$.

\begin{lemma}\label{lem:TT}
Under (T1)-(T5) we obtain: $W(\mu_{k+1}^*,\mu^*)\le \gamma W(\mu_k^*,\mu^*)$ and thus $\mu_k^*\Rightarrow \mu^*$ weakly.
\end{lemma}

Lemma \ref{lem:TT} then implies that even in this case the maximal average reward $\tilde r(\mu^*)$ is achieved by applying $\bar Q^*$ throughout the process which corresponds to a decentralized control. An example where (T1), (T3), (T4) are fulfilled is $T(x,a,\mu,z) = \gamma_S x+  \gamma_A a +\gamma_W \int x\mu(dx) + z$.

\section{Applications}\label{sec:app}
\subsection{Avoiding Congestion}
We consider here the following special case: $N$ individuals move on a graph with nodes $S=\{1,\ldots,d\}$ and edges $E\subset \{(x,x') : x,x'\in S\}$. Individuals can move along one edge in one time step. We assume that nodes are connected. The aim is to avoid congestion and to try to spread the individuals such that they keep a maximum distance. More precisely suppose that the current empirical distribution of the individuals on the nodes is $\mu$ and that the distance between node $x$ and $x'$, $x,x'\in S$ is given by $\Delta(x,x')>0$ where $\Delta(x,x)=0$ and $\Delta(x,x')=\Delta(x',x)$. Then the average distance between an individual at position $x$ and all other individuals is 
$$ r(x,a,\mu) = r(x,\mu) = \sum_{x'} \Delta(x,x') \mu(x')= \int \Delta(x,x')\mu(dx').$$
Here $r(x,a,\mu)$ does not depend on a. Hence
$$ \tilde r(\mu,Q) = \tilde r(\mu) = \int r(x,\mu) \mu(dx) = \int\int \Delta(x,x') \mu(dx)\mu(dx')= \mu \Delta \mu^\top $$
where $\Delta=\big( \Delta(x,x')\big)_{x,x'\in S}$ is the matrix of distances. Note that $\Delta$ is symmetric. We assume that $A=S$ and $D(x)=\{x'\in S : (x,x')\in E\}\cup\{x\}$, i.a.\ actions in the original model are neighbours on the graph. We interpret actions as intended directions the individual wants to move to, but this may be disturbed by some random external noise.  In the mean-field limit the state of the system at time $n$ is just given by a generalized distribution on $S$. Recall that the general transition equation of the mean-field limit is
\begin{eqnarray} \nonumber \mu_{n+1}(x') &=& \sum_x \sum_{a\in D(x)} p^{x,a,\mu_n,z^0}(x') Q_n(x,a) \\ \label{eq:mugraph} 
 &=& \sum_x \sum_{a\in D(x)} p^{x,a,\mu_n,z^0}(x') \bar{Q}_n(a|x)\mu_n(x)
\end{eqnarray}
if $S,A$ are finite where $ p^{x,a,\mu,z^0}(x')  = \Pop(T(x,a,\mu,Z,z^0)=x')$ and $Q_n$ has first margin $\mu_n$.  Problems where the reward decreases when more individuals share the same state are typical for mean-field problems, see e.g.\ \cite{LYCR} where a Wardrop equilibrium is computed. In \cite{ps} the authors consider spreading contamination on graphs.

\subsubsection{No common noise}
We consider the mean-field limit now. At the beginning let us assume that $p^{x,a,\mu,z^0} = p^{x,a}$ does not depend on $\mu$ and $z^0$, i.e.\ the individuals move on their own, not affected by others and there is no common noise. Moreover, it is reasonable to set $p^{x,a}(x')=0$ if $(x,x')\notin E$ except for $x=x'$. Let us denote $  P^{\bar Q}=\big( p_{xx'}^{\bar Q}\big)$
where 
\begin{equation} \label{eq:pq} p_{xx'}^{\bar Q} = \sum_{a\in D(x)} p^{x,a}(x') \bar Q(a|x)
\end{equation}  with $\bar Q(a|x)\mu(x)=Q(x,a)$. Hence \eqref{eq:mugraph} can be written as $\mu_{n+1}=\mu_n P^{\bar Q_n}$. Here it is more intuitive to work with the conditional probabilities $\bar Q(a|x)$ instead of the joint distribution $Q(x,a)$.

Obviously the optimization problem
\begin{equation} \label{eq:optimizprobmudmu2}
\left\{ \begin{array}{ll}
\max \mu \Delta\mu^\top \\
s.t. \; \mu \in \mathbb{P}(S)
\end{array}\right.
\end{equation}
has an optimal solution $\mu^*$ since $\mathbb{P}(S)$ is compact and $\mu \Delta\mu^\top$ continuous. 

We consider the following special case: For $a,x'\in D(x)$ set $p^{x,a}(x') =\alpha$ for $a=x'$ and $p^{x,a}(x') =\frac{1-\alpha}{|D(x)|-1}$ else. All other probabilities are zero. I.e.\ if we choose a vertex $a$ we will move there with probability $\alpha$ and move to any other admissible vertex with equal probability. Formally for $x\in S$, action $a\in D(x)=\{x_1,\ldots ,x_m\}$ (where $x_i=x$ for one of the $x_i$'s) and disturbance $Z\sim U[0,1]$ the transition function in this example  is given by
$$T(x,x_i,\mu,z,z^0)= \left\{ \begin{array}{cl}
x_i, & \mbox{ if }  z\in [0,\alpha],\\
x_j, & \mbox{ if } z\in (\alpha+(j-1) \frac{1-\alpha}{m-1}, \alpha + j\frac{1-\alpha}{m-1} ],\; j=1,\ldots,i-1,\\[0.1cm]
x_j, & \mbox{ if } z\in (\alpha+(j-2) \frac{1-\alpha}{m-1}, \alpha + (j-1)\frac{1-\alpha}{m-1} ],\; j=i+1,\ldots,m.
\end{array}\right.$$

\begin{lemma}\label{lem:Qstar}
If $\mu^*(x)>0$ for all $x\in S$ and $\alpha$ is large enough, then there exists a $Q^*\in \mathbb{P}(D)$ s.t.\ $\mu^* = \mu^* P^{\bar Q^*}$, i.e. $\mu^*$ is a stationary distribution for the transition kernel $P^{\bar Q^*}$ given in \eqref{eq:pq}.
\end{lemma}

\begin{proof}
We use a construction similar to the Metropolis algorithm.
For $x,x'\in S$ let 
$$ \Psi_{xx'} := \left\{ \begin{array}{ll}
\kappa, & \mbox{ if } (x,x')\in E\\
0 & \mbox{ else.} 
\end{array}\right.$$
and
$$ p_{xx'}^{\bar Q^*} := \left\{ \begin{array}{ll}
\Psi_{xx'}\Big( \frac{\mu^*(x')}{\mu^*(x)}\wedge 1\Big), & \mbox{ if } x\neq x'\\
1- \sum_{y\neq x} \Psi_{xy} \Big( \frac{\mu^*(y)}{\mu^*(x)}\wedge 1\Big)&  \mbox{ if }  x=x'.
\end{array}\right.$$
The parameter $\kappa>0$ should be such that $P^{\bar Q^*}$ is a transition matrix.
Then the detailed balance equations 
$$ \mu^*(x) p_{xx'}^{\bar Q^*} = \mu^*(x') p_{x'x}^{\bar Q^*}, \quad x,x'\in S$$
are satisfied and hence $\mu^*$ is a stationary distribution of $P^{\bar Q^*}$.  We now have to determine $\bar Q^*$ s.t.\ $P^{\bar Q^*}$ has the specified form. Let us fix $x\in S$. We have to solve \eqref{eq:pq} for $\bar Q^*$.  We claim that \eqref{eq:pq}  is solved for 
\begin{equation}\label{eq:Qstar}
\bar Q^*(a|x)= \frac{(|D(x)|-1) p_{xa}^{\bar Q^*}-(1-\alpha)}{\alpha |D(x)| -1}.
\end{equation}
This can be seen since
\begin{eqnarray} \nonumber
 \sum_{a\in D(x)} p^{x,a}(x') \bar Q^*(a|x) &=& \bar Q^*(x'|x) \alpha +\frac{1-\alpha}{|D(x)|-1} (1-\bar Q^*(x'|x))=\\ \label{eq:pqformula}
 &=&  \bar Q^*(x'|x) \Big( \frac{\alpha |D(x)|-1}{|D(x)|-1}\Big) + \frac{1-\alpha}{|D(x)|-1}=   p_{xx'}^{\bar Q^*}.
\end{eqnarray}
In order to have $\bar Q^*(a|x)\in [0,1]$ we have to make sure that
$ \alpha\ge  p_{xx'}^{\bar Q^*} \vee (1-p_{xx'}^{\bar Q^*})$ for all $x,x'\in S$ and $\alpha\ge \frac12$.
\end{proof}

\begin{theorem}
The optimal average reward policy for the limit model considered here is the stationary policy $\psi^*=(\varphi^*,\varphi^*,\ldots)$ with $\varphi^*(\mu)= \mu \otimes \bar Q^*$ with $\bar Q^*$ from \eqref{eq:Qstar}. Thus, for $N$ large and $\beta$ close to one, sampling actions from $\bar Q^*$ is $\varepsilon$-optimal for the $\beta$-discounted problem with $N$ individuals.
\end{theorem}

\begin{proof}
The statement follows from our previous discussions. Note that when we start with an arbitrary $\mu_0^*$, the sequence of distributions generated by $\mu_{k+1}^* = \mu_k^* P^{\bar Q^*}$ converges against $\mu^*$ since the matrix $P^{\bar Q^*}$ is irreducible by construction and we have a finite state space. Thus, $G_\psi(\mu_0^*)$ in \eqref{eq:AR} yields the same limit $\mu^* \Delta (\mu^*)^\top$ which is maximal since it solves \eqref{eq:optimizprobmudmu}.
\end{proof}

\begin{remark}
It is tempting to say that for the discounted problem, once we have reached the stationary distribution after a transient phase we know that the optimal policy is to choose $\bar Q^*$ forever. However, there are only rare cases where the stationary distribution is reached after a finite number of steps (see e.g. \cite{glynn}), so the transient phase will in most cases last forever. 
\end{remark}

\begin{example}
We consider a regular $3\times 3$ grid, i.e. $d=9$ (see Figure \ref{fig:grid9}, left). We set the distance between nodes equal to 1 when there is only one edge between them. Nodes which are connected via 2 edges get the distance 1.4, when there are 3 edges in between 1.7 and finally we set the distance equal to 2.2 when there are 4 edges in between. The distance matrix $\Delta$ is thus given by
$$\Delta := \left(\begin{array}{ccccccccc}
0 & 1 & 1.4 & 1 & 1.4 & 1.7 & 1.4 & 1.7 & 2.2\\
1 & 0 & 1 & 1.4 & 1 & 1.4 & 1.7 & 1.4 & 1.7 \\
1.4 & 1 & 0 & 1.7 & 1.4 & 1 & 2.2 & 1.7 & 1.4\\
1 & 1.4 & 1.7 & 0 & 1 & 1.4 & 1 & 1.4 & 1.7\\
1.4 & 1 & 1.4 & 1 & 0 & 1 & 1.4 & 1 & 1.4\\
1.7 & 1.4 & 1 & 1.4 & 1 & 0 & 1.7 & 1.4 & 1\\
1.4 & 1.7 & 2.2 & 1 & 1.4 & 1.7 & 0 & 1 & 1.4\\
1.7 & 1.4 & 1.7 & 1.4 & 1 & 1.4 & 1 & 0 & 1\\
2.2 & 1.7 & 1.4 & 1.7 & 1.4 & 1 & 1.4 & 1 & 0
\end{array}\right)$$
The optimal distribution of problem \eqref{eq:optimizprobmudmu} is here given by $\mu^*= \frac{1}{37} (7,2,7,2,1,2,7,2,7)$. The masses are illustrated in Figure \ref{fig:grid9}, right picture. The area of the circle is proportional to the corresponding value of $\mu^*$. We think of the proportion of individuals who occupy this node.

   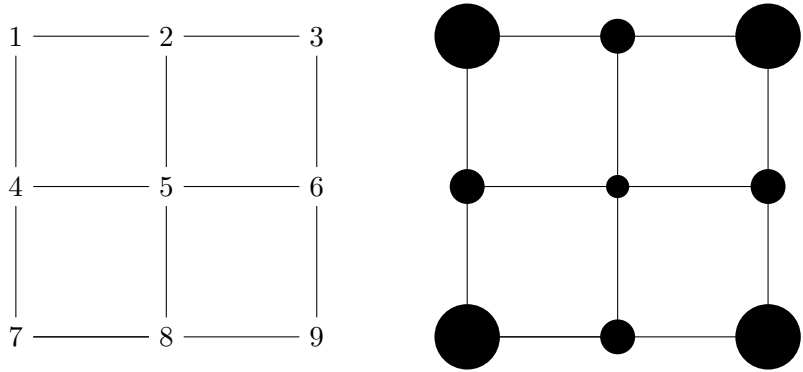
\begin{figure}[!htbp]\begin{center}
\begin{tikzpicture}

 \node (v0) at (0,0) {$7$};
  \node (v1) at (0,2) {$4$};
   \node (v2) at (0,4) {$1$};
    \node (v3) at (2,0) {$8$};
     \node (v4) at (2,2) {$5$};
      \node (v5) at (2,4) {$2$};
       \node (v6) at (4,0) {$9$};
        \node (v7) at (4,2) {$6$};
         \node (v8) at (4,4) {$3$};
         \draw [-] (v0) -- (v3);
             \draw [-] (v0) -- (v3);    
             \draw [-] (v0) -- (v1);    
             \draw [-] (v3) -- (v6);    
             \draw [-] (v3) -- (v4);    
             \draw [-] (v6) -- (v7);   
              \draw [-] (v1) -- (v4);    
              \draw [-] (v4) -- (v7);    
              \draw [-] (v1) -- (v2);    
              \draw [-] (v4) -- (v5);    
              \draw [-] (v7) -- (v8);    
              \draw [-] (v8) -- (v5);
                  \draw [-] (v2) -- (v5);
 
 \fill (6,0) circle (0.435);
\fill (6,2) circle (0.232);
\fill (6,4) circle (0.435);
\fill (8,0) circle (0.232);
\fill (8,2) circle (0.154);
\fill (8,4) circle (0.232);
\fill (10,0) circle (0.435);
\fill (10,2) circle (0.232);
\fill (10,4) circle (0.435);                 
   \node (v9) at (6,0) {};
  \node (v10) at (6,2) {};
   \node (v11) at (6,4) {};
    \node (v12) at (8,0) {};
     \node (v13) at (8,2) {};
      \node (v14) at (8,4) {};
       \node (v15) at (10,0) {};
        \node (v16) at (10,2) {};
         \node (v17) at (10,4) {};
         \draw [-] (v9) -- (v12);
             \draw [-] (v9) -- (v12);    
             \draw [-] (v9) -- (v10);    
             \draw [-] (v12) -- (v15);    
             \draw [-] (v12) -- (v13);    
             \draw [-] (v15) -- (v16);   
              \draw [-] (v10) -- (v13);    
              \draw [-] (v13) -- (v16);    
              \draw [-] (v10) -- (v11);    
              \draw [-] (v13) -- (v14);    
              \draw [-] (v16) -- (v17);    
              \draw [-] (v17) -- (v14);
                  \draw [-] (v11) -- (v14);
  
\end{tikzpicture}
  \caption{Network with labelled nodes (left); Optimal stationary distribution (right).}
         \label{fig:grid9}
\end{center}
 \end{figure}
We set $\alpha=1$ and $\psi=0.25$. Then we obtain from \eqref{eq:Qstar} that the optimal  decision in every node is given by the following transition kernel $\bar Q^*(a|x)$ 
$$\bar Q^* := \left(\begin{array}{ccccccccc}
12c & c & 0 & c & 0 & 0 & 0 & 0 & 0\\
2b & 3b & 2b & 0 & b & 0 & 0 & 0 & 0 \\
0 & c & 12c & 0 & 0 & c & 0 & 0 & 0\\
2b & 0 & 0 & 3b & b & 0 & 2b & 0 & 0\\
0 & 2b & 0 & 2b & 0 & 2b & 0 & 2b & 0\\
0 & 0 & 2b & 0 & b & 3b & 0 & 0 & 2b\\
0 & 0 & 0 & c & 0 & 0 & 12c & c & 0\\
0 & 0 & 0 & 0 & b & 0 & 2b & 3b & 2b\\
0 & 0 & 0 & 0 & 0 & c & 0 & c & 12c
\end{array}\right)$$
where $b=\frac{1}{8}$ and $c=\frac{1}{14}$. So using this decentralized decision throughout the process yields the maximal average reward. In Figure \ref{fig:gridevolution} we see the evolution of the system when all mass starts initially in node 1. The pictures show the distribution of the mass after 2, 4, 8, 16, 32 and 64 time steps. Note that sampling actions from $\bar Q^*$ is also $\varepsilon$-optimal for the system when we have a finite but large number of individuals and $\beta$ is close to one for the discounted reward criterion.

   \begin{figure}[!htbp]
\begin{center}
\begin{tikzpicture}
\fill (0,0) circle (0.134);
\fill (0,1.5) circle (0.297);
\fill (0,3) circle (0.878);
\fill (1.5,0) circle (0);
\fill (1.5,1.5) circle (0.134);
\fill (1.5,3) circle (0.297);
\fill (3,0) circle (0);
\fill (3,1.5) circle (0);
\fill (3,3) circle (0.134);

 \node (v0) at (0,0) {};
  \node (v1) at (0,1.5) {};
   \node (v2) at (0,3) {};
    \node (v3) at (1.5,0) {};
     \node (v4) at (1.5,1.5) {};
      \node (v5) at (1.5,3) {};
       \node (v6) at (3,0) {};
        \node (v7) at (3,1.5) {};
         \node (v8) at (3,3) {};
         \draw [-] (v0) -- (v3);
             \draw [-] (v0) -- (v3);    
             \draw [-] (v0) -- (v1);    
             \draw [-] (v3) -- (v6);    
             \draw [-] (v3) -- (v4);    
             \draw [-] (v6) -- (v7);   
              \draw [-] (v1) -- (v4);    
              \draw [-] (v4) -- (v7);    
              \draw [-] (v1) -- (v2);    
              \draw [-] (v4) -- (v5);    
              \draw [-] (v7) -- (v8);    
              \draw [-] (v8) -- (v5);
                  \draw [-] (v2) -- (v5);
     \fill (5,0) circle (0.238);
\fill (5,1.5) circle (0.306);
\fill (5,3) circle (0.807);
\fill (6.5,0) circle (0.102);
\fill (6.5,1.5) circle (0.158);
\fill (6.5,3) circle (0.306);
\fill (8,0) circle (0.054);
\fill (8,1.5) circle (0.102);
\fill (8,3) circle (0.238);
             
   \node (v9) at (5,0) {};
  \node (v10) at (5,1.5) {};
   \node (v11) at (5,3) {};
    \node (v12) at (6.5,0) {};
     \node (v13) at (6.5,1.5) {};
      \node (v14) at (6.5,3) {};
       \node (v15) at (8,0) {};
        \node (v16) at (8,1.5) {};
         \node (v17) at (8,3) {};
        
             \draw [-] (v9) -- (v12);    
             \draw [-] (v9) -- (v10);    
             \draw [-] (v12) -- (v15);    
             \draw [-] (v12) -- (v13);    
             \draw [-] (v15) -- (v16);   
              \draw [-] (v10) -- (v13);    
              \draw [-] (v13) -- (v16);    
              \draw [-] (v10) -- (v11);    
              \draw [-] (v13) -- (v14);    
              \draw [-] (v16) -- (v17);    
              \draw [-] (v17) -- (v14);
                  \draw [-] (v11) -- (v14);
      \fill (10,0) circle (0.341);
\fill (10,1.5) circle (0.291);
\fill (10,3) circle (0.704);
\fill (11.5,0) circle (0.153);
\fill (11.5,1.5) circle (0.164);
\fill (11.5,3) circle (0.291);
\fill (13,0) circle (0.168);
\fill (13,1.5) circle (0.153);
\fill (13,3) circle (0.341);
             
   \node (v18) at (10,0) {};
  \node (v19) at (10,1.5) {};
   \node (v20) at (10,3) {};
    \node (v21) at (11.5,0) {};
     \node (v22) at (11.5,1.5) {};
      \node (v23) at (11.5,3) {};
       \node (v24) at (13,0) {};
        \node (v25) at (13,1.5) {};
         \node (v26) at (13,3) {};
      
             \draw [-] (v18) -- (v21);    
             \draw [-] (v18) -- (v19);    
             \draw [-] (v21) -- (v24);    
             \draw [-] (v21) -- (v22);    
             \draw [-] (v24) -- (v25);   
              \draw [-] (v19) -- (v22);    
              \draw [-] (v22) -- (v25);    
              \draw [-] (v19) -- (v20);    
              \draw [-] (v22) -- (v23);    
              \draw [-] (v25) -- (v26);    
              \draw [-] (v26) -- (v23);
                 \draw [-] (v20) -- (v23);
           \fill (0,-4) circle (0.41);        
\fill (0,-2.5) circle (0.264);
\fill (0,-1) circle (0.577);
\fill (1.5,-4) circle (0.195);
\fill (1.5,-2.5) circle (0.164);
\fill (1.5,-1) circle (0.264);
\fill (3,-4) circle (0.269);
\fill (3,-2.5) circle (0.195);
\fill (3,-1) circle (0.41);

 \node (v27) at (0,-4) {};
  \node (v28) at (0,-2.5) {};
   \node (v29) at (0,-1) {};
    \node (v30) at (1.5,-4) {};
     \node (v31) at (1.5,-2.5) {};
      \node (v32) at (1.5,-1) {};
       \node (v33) at (3,-4) {};
        \node (v34) at (3,-2.5) {};
         \node (v35) at (3,-1) {};
         \draw [-] (v27) -- (v30);
             \draw [-] (v27) -- (v30);    
             \draw [-] (v27) -- (v28);    
             \draw [-] (v30) -- (v33);    
             \draw [-] (v30) -- (v31);    
             \draw [-] (v33) -- (v34);   
              \draw [-] (v28) -- (v31);    
              \draw [-] (v31) -- (v34);    
              \draw [-] (v28) -- (v29);    
              \draw [-] (v31) -- (v32);    
              \draw [-] (v34) -- (v35);    
              \draw [-] (v35) -- (v32);
                  \draw [-] (v29) -- (v32);      
    \fill (5,-4) circle (0.433);
\fill (5,-2.5) circle (0.242);
\fill (5,-1) circle (0.474);
\fill (6.5,-4) circle (0.223);
\fill (6.5,-2.5) circle (0.164);
\fill (6.5,-1) circle (0.242);
\fill (8,-4) circle (0.397);
\fill (8,-2.5) circle (0.223);
\fill (8,-1) circle (0.433);
             
   \node (v36) at (5,-4) {};
  \node (v37) at (5,-2.5) {};
   \node (v38) at (5,-1) {};
    \node (v39) at (6.5,-4) {};
     \node (v40) at (6.5,-2.5) {};
      \node (v41) at (6.5,-1) {};
       \node (v42) at (8,-4) {};
        \node (v43) at (8,-2.5) {};
         \node (v44) at (8,-1) {};
        
             \draw [-] (v36) -- (v39);    
             \draw [-] (v36) -- (v37);    
             \draw [-] (v39) -- (v42);    
             \draw [-] (v39) -- (v40);    
             \draw [-] (v42) -- (v43);   
              \draw [-] (v37) -- (v40);    
              \draw [-] (v40) -- (v43);    
              \draw [-] (v37) -- (v38);    
              \draw [-] (v40) -- (v41);    
              \draw [-] (v43) -- (v44);    
              \draw [-] (v44) -- (v41);
                  \draw [-] (v38) -- (v41);
    \fill (10,-4) circle (0.435);
\fill (10,-2.5) circle (0.233);
\fill (10,-1) circle (0.438);
\fill (11.5,-4) circle (0.231);
\fill (11.5,-2.5) circle (0.164);
\fill (11.5,-1) circle (0.233);
\fill (13,-4) circle (0.432);
\fill (13,-2.5) circle (0.231);
\fill (13,-1) circle (0.435);
             
   \node (v45) at (10,-4) {};
  \node (v46) at (10,-2.5) {};
   \node (v47) at (10,-1) {};
    \node (v48) at (11.5,-4) {};
     \node (v49) at (11.5,-2.5) {};
      \node (v50) at (11.5,-1) {};
       \node (v51) at (13,-4) {};
        \node (v52) at (13,-2.5) {};
         \node (v53) at (13,-1) {};
      
             \draw [-] (v45) -- (v48);    
             \draw [-] (v45) -- (v46);    
             \draw [-] (v48) -- (v51);    
             \draw [-] (v48) -- (v49);    
             \draw [-] (v51) -- (v52);   
              \draw [-] (v46) -- (v49);    
              \draw [-] (v49) -- (v52);    
              \draw [-] (v46) -- (v47);    
              \draw [-] (v49) -- (v50);    
              \draw [-] (v52) -- (v53);    
              \draw [-] (v53) -- (v50);
                 \draw [-] (v47) -- (v50);
\end{tikzpicture}
  \caption{Evolution of the individuals using the optimal randomized decision when all start in node 1, after $n=2,4,8,16,32$ and $64$ time steps (left to right, above to below).}
         \label{fig:gridevolution}
\end{center}
   \end{figure}
\end{example}

\subsubsection{With common noise}
Next we suppose that $\alpha$ depends on the common noise $Z^0$.  In this case the maximal average reward which can be achieved is less or equal to the case without common noise since the sequence of distributions is stochastic and may deviate from the optimal one. We simplify things a little bit since we assume here that $|D(x)| = \gamma$ independent of $x$. From the previous section, equation \eqref{eq:pqformula} we know that we can write 
$$ p^{\bar Q}_{xx'} = \bar Q(x'|x) \frac{\alpha(Z^0)\gamma-1}{\gamma-1}+\frac{1-\alpha(Z^0)}{\gamma-1}. $$
In matrix notation $$P^{\bar Q} = \frac1{\gamma-1} (1-\alpha(Z^0)) U + \frac1{\gamma-1} (\alpha(Z^0)\gamma-1)\bar Q$$ where $U$ is a $d\times d$ matrix containing ones only and $\bar Q=(\bar Q(x'|x)).$ Here the situation is more complicated, in particular the next empirical distribution of individuals is stochastic and given by
$$\mu_{n+1}= \frac1{\gamma-1} (1-\alpha(Z^0)) e +\frac1{\gamma-1}  (\alpha(Z^0)\gamma-1)\mu_n \bar Q_n $$
with $e=(1,\ldots,1)\in \R^d$. Plugging this into the reward function yields

\begin{eqnarray}\nonumber
{(\gamma-1)^2}\Eop\Big[ \mu_{n+1} \Delta \mu_{n+1}^\top\Big] &=& \Eop[(1-\alpha(Z^0))^2] e \Delta e^\top+ 2 \Eop[(1-\alpha(Z^0)) ( \alpha(Z^0)\gamma-1)]  (e \Delta \bar Q_n^\top \mu_n^\top) \\ \label{eq:target}
&&+ \Eop[(\alpha(Z^0)\gamma-1)^2] \mu_n \bar Q_n \Delta \bar Q_n^\top \mu_n^\top.
\end{eqnarray}
Now consider the problem
\begin{equation}
\left\{ \begin{array}{l}
 2 \Eop[(1-\alpha(Z^0)) ( \alpha(Z^0)\gamma-1)]  (e \Delta \nu^\top) + \Eop[(\alpha(Z^0)\gamma-1)^2] \nu \Delta \nu^\top \to \max\\
\nu\in \mathbb{P}(S)
\end{array} \right.
\end{equation}
Obviously this problem has an optimal solution  $\nu^*$ since we maximize a continuous function over a compact set. Now $\nu$ corresponds to $\mu_n \bar Q_n$ in  \eqref{eq:target}. In case it is possible to choose for all $\mu\in \mathbb{P}(S)$ a matrix $\bar Q$ s.t. $\mu \bar Q= \nu^*$, then this would be the optimal strategy, since we would get the maximal expected reward in each step. This is for example possible if the graph is complete. Then we can simply choose $\bar Q$ as the matrix with identical rows which consist of $\nu^*$. 
 
\subsection{Positioning on a Market Place}
Suppose we have a rectangular market place like in Figure \ref{fig:market}. The state $\mu$ represents the distribution of individuals over the market place. Point A is an ice cream vendor. The aim of the individuals is to keep distance to others and be as close as possible to the ice cream vendor. Thus, $S\subset \R^2$ is the rectangle $BCED$ and the one-stage reward is
$$\tilde r(\mu)= \int\int d(x,y)\mu(dx)\mu(dy) -\int d(x,A) \mu(dx).$$ 
In what follows in order to simplify the computation we choose $d(x,y)=\|x-y\|^2$ for $x,y\in S$. We want to solve \eqref{eq:optimizprobmudmu} in this case. Let us formulate the problem with the help of random variables. Let $X=(X_1,X_2), Y=(Y_1,Y_2)$ be independent r.v.\ having distribution $\mu$. Then $\tilde r(\mu)$ is the same as
$$ \sum_{i=1}^2 \Eop(X_i-Y_i)^2 - \Eop(X_i-A_i)^2.$$
Thus, we can treat the margins separately and the dependence between them is not interesting for the reward. Now obviously since $X$ and $Y$ both have the same distribution we can write
\begin{eqnarray*}
\Eop(X_i-Y_i)^2 - \Eop(X_i-A_i)^2 &=& \Eop X_i^2 + 2 \Eop X_i (A_i-\Eop X_i) - A_i^2.
\end{eqnarray*}
Suppose we fix $\Eop X_i$ for a moment. Since $x\mapsto x^2$ is convex, the distribution which maximizes the expression is maximal in convex order, given the fixed expectation. But this distribution is due to the convexity property concentrated on the endpoints of the interval. Thus we can restrict to random variables $X_1$ which have mass $p\in [0,1]$ on $B_1$ and $1-p$ on $C_1$, i.e. we maximize
$$B_1^2 p+C_1^2 (1-p)+2(B_1 p+C_1 (1-p)) (A_1-B_1p-C_1 (1-p))$$ over $p\in [0,1]$. 

   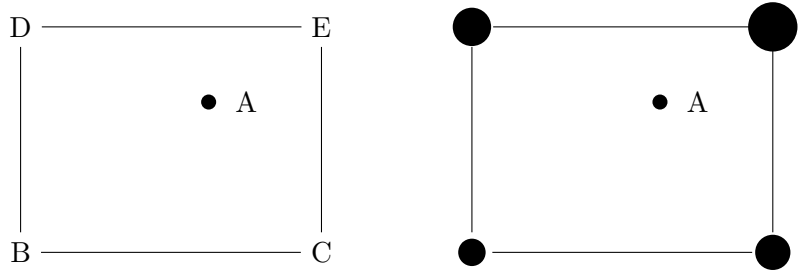
\begin{figure}[!htbp]\begin{center}
\begin{tikzpicture}

 \node (v0) at (0,0) {B};
  \node (v1) at (0,3) {D};
   \node (v2) at (4,0) {C};
    \node (v3) at (4,3) {E};
 
         \node (v8) at (3,2) {A};
         \draw [-] (v0) -- (v1);
             \draw [-] (v0) -- (v2);    
             \draw [-] (v3) -- (v1);    
             \draw [-] (v3) -- (v2);

 \fill (2.5,2) circle (0.1);
 
  \node (v0) at (6,0) {B};
  \node (v1) at (6,3) {D};
   \node (v2) at (10,0) {C};
    \node (v3) at (10,3) {E};
 
         \node (v8) at (9,2) {A};
         \draw [-] (v0) -- (v1);
             \draw [-] (v0) -- (v2);    
             \draw [-] (v3) -- (v1);    
             \draw [-] (v3) -- (v2);

 \fill (8.5,2) circle (0.1);
 
  \fill (6,0) circle (35/192);
   \fill (10,0) circle (45/192);
    \fill (6,3) circle (49/192);
     \fill (10,3) circle (63/192);
  
\end{tikzpicture}
  \caption{Market place with ice cream vendor (left). Optimal distribution in example (right)}
         \label{fig:market}
\end{center}
 \end{figure}
 
 The solution is given by $p= \frac14 + \frac{C_1-A_1}{2(C_1-B_1)}$. Since the joint distribution does not matter we can choose independent margins and obtain 
 \begin{eqnarray*}
  \mu^* &=& \delta_B \Big(  \frac14 + \frac{C_1-A_1}{2(C_1-B_1)}\Big)\Big(  \frac14 + \frac{D_2-A_2}{2(D_2-B_2)}\Big) + \delta_C  \Big(  \frac34 - \frac{C_1-A_1}{2(C_1-B_1)}\Big)\Big(  \frac14 + \frac{D_2-A_2}{2(D_2-B_2)}\Big)\\
  && +  \delta_D \Big(  \frac14 + \frac{C_1-A_1}{2(C_1-B_1)}\Big)\Big(  \frac34 - \frac{D_2-A_2}{2(D_2-B_2)}\Big)+\delta_E  \Big(  \frac34 - \frac{C_1-A_1}{2(C_1-B_1)}\Big)\Big(  \frac34 - \frac{D_2-A_2}{2(D_2-B_2)}\Big).
\end{eqnarray*} 
This is the target distribution which should be attained. For a numerical example we choose $B(0,0), C(4,0), D(0,3), E(4,3)$ and $A(2.5,2)$. In this case we obtain
$$ \mu^*= \delta_B \frac{35}{192}+\delta_C \frac{45}{192}+\delta_D \frac{49}{192}+\delta_E \frac{63}{192}.$$
The distribution is illustrated in Figure \ref{fig:market}, (right).

  Depending on how the transition law precisely looks like, if one is able to choose $\bar Q^*$ such that $\mu^*$ is the stationary distribution of $P^{\bar Q^*}$, the problem is solved. Of course the optimal distribution $\mu^*$ depends on what kind of distance $d$ we choose. Varying  the metric for the distance leads to interesting optimization problems.

\section{Conclusion}
We have seen that the average reward mean-field problem can in some cases be solved rather easily by computing an optimal measure from a static optimization problem. The policy which is obtained in this way is $\varepsilon$-optimal for the $\beta$-discounted $N$-individuals problem where $N$ is large and $\beta$ close to one. The static optimization problem for measures gives rise to some interesting mathematical questions.

\section{Appendix}\label{sec:appendix}
\subsection{Auxiliary Results}

The following result can be found in \cite{bl}, Lemma 7.2:

\begin{lemma}\label{lem:unifconf}
Let $X$ be a separable metric space, $Y$ be compact metric and $f:X \times Y\to \R$ continuous. Then $x_n\to x$ for $n\to\infty$ implies
$$\lim_{n\to \infty} \sup_{y\in Y} |f(x_n,y)-f(x,y)|=0. $$ 
\end{lemma}

\subsection{Wasserstein Ergodicity}
For the following definitions and results see \cite{rs}.

\begin{definition}
For two probability measures $\mu,\nu$ on $S$, the dual representation of the Wasserstein distance is given by
$$ W(\mu,\nu) = \sup_{\|f\|_{Lip}\le 1} \Big| \int f(x) (\nu(dx)-\mu(dx))\Big|$$
where 
$$ \|f\|_{Lip} = \sup_{x,y\in S, x\neq y} \frac{|f(x)-f(y)|}{|x-y|}.$$
\end{definition}
Note that convergence in Wasserstein metric implies weak convergence when we are on compact sets.

\begin{definition}
A transition kernel $P(\cdot| x)$ from $S$ to $S$ is called {\em Wasserstein ergodic} when there exist constants $\rho\in (0,1)$ and $C>0$ s.t. for all $n\in\N$
$$ \sup_{x,y\in S, x\neq y} \frac{W(P^n(\cdot|x), P^n(\cdot|y))}{|x-y|}\le C \rho^n. $$
\end{definition}
Suppose $P$ is Wasserstein ergodic and has stationary distribution $\mu^*$ which means that  $\mu^* = \int P(\cdot|x)\mu^*(dx)=:\mu^* P$. Then for any $\mu_0\in \mathbb{P}$ and $\mu_n = \mu_0  P^n$ we obtain $W(\mu_n,\mu^*) \le C \rho^n$.

\subsection{Additional Proofs}

\subsubsection{ Proof of Theorem \ref{thm:model1}:}

We first show that $U:\mathbb{M}\to\mathbb{M}$. Hence, let $v\in \mathbb{M}$.
Since $r$ and $v$ are bounded, $Uv$ is bounded. (A2) implies that $(\mathbf{x},\mathbf{a}) \mapsto \mathbf{r}(\mathbf{x},\mathbf{a})$ is upper semicontinuous. This follows since $(\mathbf{x}_n,\mathbf{a}_n) \to (\mathbf{x},\mathbf{a})$ for $n\to\infty$ implies $x_n^i\to x^i, a_n^i \to a^i$, $i=1,\ldots,N$ and  $\mu[\mathbf{x}_n] \to \mu[\mathbf{x}]$ (in weak topology) for $n\to\infty$. Moreover, the sum of upper semicontinuous functions is upper semicontinuous.
And finally due to (A3) and the fact that $v$ is upper semicontinuous
$$  (\mathbf{x},\mathbf{a})\mapsto  \Eop\Big[v\big(  \mathbf{T}(\mathbf{x}, \mathbf{a}, \mu[\mathbf{x}], \mathbf{Z}, Z^0)\big)\Big]\Big\}$$
is upper semicontinuous. This together implies that
\begin{equation}\label{eq:rvmap}
 (\mathbf{x},\mathbf{a})\mapsto  \mathbf{r}(\mathbf{x},\mathbf{a})+\Eop\Big[v\big(  \mathbf{T}(\mathbf{x}, \mathbf{a}, \mu[\mathbf{x}], \mathbf{Z}, Z^0)\big)\Big]\Big\}
 \end{equation} 
is upper semicontinuous and $U:\mathbb{M}\to\mathbb{M}$ follows from Proposition 2.4.3 in \cite{br11}.

Next note  that  $\mathbb{M}$  together with the sup-norm $\| v\| = \sup_{\mathbf{x}\in S^N} |v(\mathbf{x})|$  is a Banach space. Also $0\in \mathbb{M}$ which is the function identical to zero. Moreover for $v,w\in \mathbb{M}$:
\begin{eqnarray*}
\| Uv-Uw\| &\le & \beta \sup_{\mathbf{a}\in D(\mathbf{x})} \Big\{ \Eop\Big[v\big(  \mathbf{T}(\mathbf{x}, \mathbf{a}, \mu[\mathbf{x}], \mathbf{Z}, Z^0)\big)-w\big(  \mathbf{T}(\mathbf{x}, \mathbf{a}, \mu[\mathbf{x}], \mathbf{Z}, Z^0)\big)\Big] \Big\} \\
&\le& \beta \|v-w\|
\end{eqnarray*} 
thus $U$ is contracting since $\beta\in(0,1)$.  Next, the properties in (A0), (A1) imply  that $\mathbf{D}(\mathbf{x})$ is compact and $\mathbf{x}\mapsto \mathbf{D}(\mathbf{x})$ is upper semicontinuous. From the first part of the proof we know that the mapping in \eqref{eq:rvmap} is upper semicontinuous. Thus,  the existence result for maximizers from Proposition 2.4.3 in \cite{br11} implies that for all $v\in \mathbb{M}$ there exists a maximizer $f\in F$.

 Altogether, we have shown all assumptions from Theorem 7.3.5  in \cite{br11} which directly implies the statement.\hfill$\Box$

\subsubsection{ Proof of Theorem \ref{theo:JN}:}
We only have to show that $\hat U :\mathbb{\hat M}\to \mathbb{\hat M}$.
The statement then follows from Theorem \ref{thm:model1} and  Theorem \ref{thm:equality} since we can identify policies, rewards, transition laws and operators. 
 To show $\hat U :\mathbb{\hat M}\to \mathbb{\hat M}$ we use Proposition 2.4.3 in \cite{br11}. Thus, we have to check the following continuity and compactness assumptions. 
\begin{itemize}
\item[(i)] $\hat D(\mu)$ is compact and $\mu \mapsto \hat D(\mu)$ is upper semicontinuous on $\Pop_N(S)$.
\item[(ii)] $(\mu,Q)\mapsto \hat r(\mu,Q)$ is upper semicontinuous and bounded on $\hat D$.
\item[(iii)] for $v\in \mathbb{\hat M}$ the mapping $(\mu,Q)\mapsto  \Eop v( \hat{ T}(\mu,Q,\mathbf{Z},Z^0))$ is upper semicontinuous and bounded on $\hat D$.
\end{itemize}

For (i) first note that $\hat D(\mu)$ is compact for all $\mu$ since $D$ is compact.  Upper semicontinuity of  $\mu \mapsto \hat D(\mu)$ can be seen as follows: Let $(\mu_n) \subset \Pop_N(S)$ and $\mu_n \Rightarrow \mu$ for $n\to\infty$ and $Q_n\in \hat D(\mu_n)$. Since $\hat D$ is compact there exists an accumulation point $Q\in \mathbb{P}_N(D)$ s.t. $Q_{n_k}\Rightarrow Q$ for a subsequence $(n_k),$ and the sequence of the first margins converges to $\mu$ hence $Q\in \hat D(\mu)$.

Part (ii) follows from the fact that 
$$ \hat{r}(\mu,Q) =\frac1N \sum_{i=1}^N r(x_i,a_i, \mu)$$
for $Q=\mu[(\mathbf{x},\mathbf{a})]$, (A2) and the observation that $(Q_n)\subset\Pop_N(D), Q_n \Rightarrow Q\in \Pop_N(D)$ implies pointwise convergence $x_i^{(n)} \to x_i, a_i^{(n)}\to a_i$ for $n\to\infty$, $i=1,\ldots,N$.

Finally for (iii) note that $$(\mu,Q) \mapsto \hat T(\mu,Q,Z,Z^0)$$
is continuous on $\hat D$ which follows from (A3). This implies (iii) and the statement follows from Proposition 2.4.3 in \cite{br11}. \hfill$\Box$

\subsubsection{ Proof of Theorem \ref{theo:MFLexi}:}
In order to show the statement we use Theorem 7.3.5 in \cite{br11}. Thus, we  first  prove that $\tilde U:\tilde{\mathbb{M}}\to\tilde{\mathbb{M}}$. We do this by showing that
\begin{itemize}
\item[(i)] $\tilde D(\mu)$ is compact and $\mu \mapsto \tilde D(\mu)$ is continuous.
\item[(ii)] $(\mu,Q)\mapsto \tilde r(\mu,Q)$ is continuous and bounded.
\item[(iii)] for $v\in \mathbb{M}$ the mapping $(\mu,Q)\mapsto \Eop v(\tilde T(\mu,Q,Z^0))$ is continuous and bounded.
\end{itemize}

Consider (i): $\tilde D(\mu)$ is compact for all $\mu$ since $D$ is compact. Next the mapping $\mu \mapsto \tilde D(\mu)$ is continuous if and only if it is upper and lower semicontinuous.   Upper semicontinuity follows as in the proof of Lemma \ref{theo:JN}.
Lower semicontinuity means that when $\mu_n \Rightarrow \mu\in \mathbb{P}(S)$ for $n\to\infty$, then for each $Q\in \tilde D(\mu)$ we find a sequence $(Q_n)$ with  $Q_n \Rightarrow Q$ and $Q_n \in \tilde D(\mu_n)$. This can be achieved as follows: We can decompose $Q$ into $Q=\mu \otimes \bar{Q}$. Now define $Q_n := \mu_n \otimes \bar{Q}$, then the constructed sequence has the desired properties.

For (ii) suppose that  $(\mu_n, Q_n) \Rightarrow (\mu,Q)$ for $n\to\infty$. We have to show that
$$\lim_{n\to\infty }\int_{D} r(x,a,\mu_n) Q_n(d(x,a)) = \int_{D} r(x,a,\mu) Q(d(x,a)).$$  We obtain:
\begin{eqnarray*}
&&\left|  \int_{D} r(x,a,\mu_n) Q_n(d(x,a)) - \int_{D} r(x,a,\mu) Q(d(x,a))\right| \le \\
&\le & \int_{D}   \left|  r(x,a,\mu_n)-r(x,a,\mu)   \right|  Q_n(d(x,a)) +  \left|  \int_{D}   r(x,a,\mu) \big(Q_n(d(x,a))-Q(d(x,a))\big)\right|  \le \\
&\le & \sup_{(x,a)\in D} \left|  r(x,a,\mu_n)-r(x,a,\mu)   \right| +  \left|  \int_{D}   r(x,a,\mu) \big(Q_n(d(x,a))-Q(d(x,a))\big)\right|.
\end{eqnarray*}
The first term converges to zero due to  Assumption (A2') and Lemma \ref{lem:unifconf}. The second term converges to zero since $Q_n\Rightarrow Q$ for $n\to\infty$  and (A2'). Boundedness follows from the boundedness of $r$.

Next we show (iii). 
Boundedness is clear. In order to show continuity  we first consider the mapping 
\begin{equation}\label{eq:contmap}
(\mu,Q)\mapsto  \tilde{T}(\mu,Q,z^0)= \int_{ D}  p^{x,a,\mu,z^0}  Q(d(x,a))
\end{equation} 
for fixed $z^0$.
We claim that this mapping is continuous. Let $h: S\to \R$ be continuous and bounded. By $\Pop^Z$ we denote the distribution of the r.v.\ $Z_n^i$. We have to show that 
\begin{eqnarray*}
(\mu,Q) & \mapsto & \int_S h(y) \tilde T(\mu,Q,z^0)(dy)= \int_D \int_{\mathcal{Z}} h( T(x,a,\mu,z,z^0)) \Pop^Z(dz) Q(d(x,a))
\end{eqnarray*}
is a.s.\ continuous. Let $(\mu_n,Q_n) \to (\mu,Q)$. We obtain:
\begin{eqnarray*}
&& \left| \int_D \int_{\mathcal{Z}} h( T(x,a,\mu_n,z,z^0)) \Pop^Z(dz) Q_n(d(x,a)) -
\int_D \int_{\mathcal{Z}} h( T(x,a,\mu,z,z^0)) \Pop^Z(dz) Q(d(x,a))\right| \\
&\le &   \int_D \int_{\mathcal{Z}} \left| h( T(x,a,\mu_n,z,z^0))-h( T(x,a,\mu,z,z^0)) \right|   \Pop^Z(dz) Q_n(d(x,a)) \\
&& +  \left| \int_D \int_{\mathcal{Z}} h( T(x,a,\mu,z,z^0)) \Pop^Z(dz) \big( Q_n(d(x,a)) -
 Q(d(x,a))\big) \right| \\
 &\le &  \int_{\mathcal{Z}} \sup_{(x,a)\in D}\left| h( T(x,a,\mu_n,z,z^0))-h( T(x,a,\mu,z,z^0)) \right|   \Pop^Z(dz)  \\
&& +  \left| \int_D \int_{\mathcal{Z}} h( T(x,a,\mu,z,z^0)) \Pop^Z(dz) \big( Q_n(d(x,a)) -
 Q(d(x,a))\big) \right| 
\end{eqnarray*}
In the first term we can interchange the limit $\lim_{n\to\infty}$ and the integral due to dominated convergence and obtain
$$ \lim_{n\to\infty } \sup_{(x,a)\in D}\left| h( T(x,a,\mu_n,z,z^0))-h( T(x,a,\mu,z,z^0)) \right| =0$$
due to (A2') and Lemma \ref{lem:unifconf}. The second term converges to zero for $n\to\infty$ since $(x,a) \mapsto h( T(x,a,\mu,z,z^0))$ is continuous due to (A3).  In total we have shown that the mapping in \eqref{eq:contmap} is continuous. 

Finally take $v\in \tilde{\mathbb{M}}$ and  pick a sequence with $(\mu_n,Q_n)\to (\mu,Q)$ for $n\to\infty$. We obtain with dominated convergence, the continuity  of $v$ and the continuity of \eqref{eq:contmap}
\begin{eqnarray*}
&&\lim_{n\to\infty } \Eop v( \tilde{T}(\mu_n,Q_n,Z^0)) = \Eop  v( \lim_{n\to\infty }  \tilde{T}(\mu_n,Q_n,Z^0)) 
=   \Eop   v( \tilde{T}(\mu,Q_,Z^0)) 
\end{eqnarray*}
which shows the stated continuity of $(\mu,Q)\mapsto \Eop v( \tilde{T}(\mu,Q,Z^0))$. 
Now  Proposition 2.4.8 in \cite{br11} implies that $\tilde U : \tilde{\mathbb{M}} \to \tilde{\mathbb{M}}$. 

The next condition in Theorem 7.3.5 \cite{br11} is that $\tilde U$ is contracting on $ \tilde{\mathbb{M}}$. But this follows along the same lines as in the  proof of Theorem \ref{thm:model1}. Finally, the existence of maximizers which is another assumption in Theorem 7.3.5 \cite{br11} follows again from Proposition 2.4.8 in \cite{br11}.

In total the statement is a consequence of  Theorem 7.3.5 in \cite{br11} with the set $\tilde{\mathbb{M}}$.  \hfill$\Box$

\subsubsection{ Proof of Theorem \ref{thm:limit1}}

We partition the proof into three steps.

Step 1: Let $Q^N\Rightarrow Q$ for $N\to\infty$  where $ Q^N \in \mathbb{P}_N(D).$ Hence there exist $\mathbf{x}^N=(x_1^N,\ldots,x_N^N)$ and  $\mathbf{a}^N=(a_1^N,\ldots,a_N^N)\in \mathbf{D}(\mathbf{x}^N)$ s.t.\ $\mu[(\mathbf{x}^N,\mathbf{a}^N)]=Q^N$ and $\mu[\mathbf{x}^N]=\mu^N$ and $Q^N\in \hat D(\mu^N)$.

Further, suppose we fix $\omega\in \Omega$ and consider a realization $\mathbf{z}^N=(z_1^N,\ldots,z_N^N)$ of $(Z_1^N,\ldots,Z_N^N)$ and $z^0$ of $Z_1^0.$ We show that $\hat T(\mu^N, Q^N,\mathbf{z}^N,z^0) \Rightarrow \tilde T(\mu,Q,z^0)$ where $\mu$ is the first margin of $Q$. In order to show this let $h:S\to\R$ be bounded and continuous. We obtain:
\begin{eqnarray}\nonumber
\int h(y) \hat T(\mu^N, Q^N,\mathbf{z}^N,z^0)(dy) &=& \frac1N \sum_{i=1}^N h\big(T(x_i^N,a_i^N,\mu^N,z_i^N,z^0)\big)\\ \nonumber
&=& \frac1N \sum_{i=1}^N \Big( h\big(T(x_i^N,a_i^N,\mu^N,z_i^N,z^0)\big)-h\big(T(x_i^N,a_i^N,\mu,z_i^N,z^0)\big)\Big) \\ \label{eq:hdiff}
&& + \frac1N \sum_{i=1}^N h\big(T(x_i^N,a_i^N,\mu,z_i^N,z^0)\big).
\end{eqnarray}
Since $h,T$ are continuous, $D, \mathcal{Z}$ are compact and $\mu^N\Rightarrow \mu$ we can for all $\varepsilon>0$ choose $N$ large enough s.t.
$$ \sup_{(x,a,z,z^0)\in D\times \mathcal{Z}^2} | h\big( T(x,a,\mu^N,z,z^0)\big) - h\big(T(x,a,\mu,z,z^0)\big)|\le\varepsilon.$$
Hence the first term in \eqref{eq:hdiff} converges to zero for $N\to\infty$. Let $\mu^N_z$ be the empirical measures of $\mathbf{z}^N$. We obtain:
\begin{eqnarray}\label{eq:empm}
\frac1N \sum_{i=1}^N h\big(T(x_i^N,a_i^N,\mu,z_i^N,z^0)\big)&=&
\int h\big( T(x,a,\mu,z,z^0)\big) Q^N(d(x,a)) \mu_z^N(dz).
\end{eqnarray}
 Since $Q^N \otimes \mu_z^N \Rightarrow Q\otimes \mathbb{P}^Z$ for $N\to\infty$ by the Glivenko-Cantelli Theorem for $N\to\infty$, the r.h.s.\ of   \eqref{eq:empm} converges to 
$$\int h\big( T(x,a,\mu,z,z^0)\big) Q(d(x,a)) \Pop^Z(dz) =\int h(y) \tilde T(\mu,Q,z^0)(dy).$$
Thus, we get $\hat T(\mu^N, Q^N,\mathbf{Z}^N,Z^0) \Rightarrow \tilde T(\mu,Q,Z^0)$ $\Pop$-a.s. In the proof of Theorem \ref{theo:MFLexi} we have shown that this implies $\lim_{N\to\infty} \tilde r(\mu^N,Q^N) = \tilde r(\mu,Q)$.

Step 2: Suppose $\psi^N=(\varphi_0^N,\varphi_1^N,\ldots)$ is an arbitrary policy for  $\widehat{\rm MDP}$. Let $Q_0^N = \varphi_0^N(\mu_0^N)$. Now $(Q_0^N)$ is a sequence of measures on the compact space $D$. Hence there is a subsequence $(m_N)$ s.t.\ $Q_0^{m_N}\Rightarrow Q_0\in \Pop(D)$ for $N\to\infty$. From Step 1 we know that  $\lim_{N\to\infty} \tilde r(\mu_0^{m_N},Q_0^{m_N}) = \tilde r(\mu_0,Q_0)$ where $\mu_0$ is the first margin of $Q_0$ and that
\begin{equation}\label{eq:mulim}
\mu_1^{m_N}=\hat T(\mu_0^{m_N},Q_0^{m_N},\mathbf{Z}_1,Z_1^0) \Rightarrow \tilde T(\mu_0,Q_0,Z_1^0), \; \Pop-\rm{a.s}.
\end{equation} 
Let $Q_1^{m_N}=\varphi_1(\mu_1^{m_N})$ and choose again a subsequence $m_N'$ s.t.\ $Q_1^{m_N'} \Rightarrow Q_1$ where the first margin of $Q_1$ is $\tilde T(\mu_0,Q_0,Z_1^0)$. When we consider the first $L\in \N$ transitions in that way, we find a joint subsequence (for convenience still denoted by $m_N$) s.t.\ for $N\to\infty$ $\Pop$-a.s.
$$(\mu_0^{m_N}, Q_0^{m_N},\mu_1^{m_N}, Q_1^{m_N},\ldots,\mu_L^{m_N}, Q_L^{m_N}) \Rightarrow (\mu_0, Q_0, \mu_1,Q_1,\ldots,\mu_L,Q_L)  $$
and where the limit is by construction an admissible state-action sequence for $\widetilde{\rm MDP}$. This is because the subsequences are taken such that the limits satisfy $Q_n\in \Pop(D)$ that the first margin of $Q_n$ is $\mu_n$ and finally because of \eqref{eq:mulim} which is by induction not only satisfied for time point one, but also for $n=1,\ldots,L$.  Hence
$$ \lim_{N\to\infty}\sum_{k=0}^L \beta^k \Eop[ \tilde r(\mu_k^{m_N},Q_k^{m_N})] = \sum_{k=0}^L \beta^k \Eop[ \tilde r(\mu_k,Q_k)]. $$
Since $|r|\le C$ we can choose $L$ large enough s.t.\ 
$$ \sum_{k=L+1}^\infty \beta^k \Eop[ \tilde r(\mu_k^{m_N},Q_k^{m_N})] \le C \frac{\beta^{L+1}}{1-\beta}.$$ This implies $\limsup_{N\to\infty} J^N(\mu^N_0)\le J(\mu_0)$.

Step 3: We finally have to show that we can construct from $\varphi^*$ a policy $\psi^N=(\varphi_0^N,\varphi_1^N,\ldots)$ s.t.\ $\limsup_{N\to\infty} J^N(\mu^N_0)= J(\mu_0)$. This proves a) and b). Suppose $\varphi^*(\mu_0)=Q_0^*$. It is possible to construct a sequence $Q_0^N\in \Pop_N(D)$ s.t.\ $Q_0^N \Rightarrow Q_0^*$ and $\mu_0^N$ is the first margin of $Q_0^N$. This can be done as follows: Suppose $Q^*_0 = \mu_0\otimes \bar Q_0$ then $\mu_0^N\Rightarrow \mu_0$ by assumption and $Q_0^N = \mu_0^N \otimes \bar Q_0^N$ where the kernel $\bar Q_0^N$ is an appropriate discretization of $\bar Q_0$ (e.g. by quantization or quasi Monte Carlo methods). Applying the results in Step 1 we obtain $\lim_{N\to\infty} \tilde r(\mu_0^N,Q_0^N) = \tilde r(\mu^*,Q^*)$ and $\mu_1^N=\hat T(\mu_0^N, Q_0^N,\mathbf{Z}_1,Z^0_1) \Rightarrow \tilde T(\mu^*,Q^*,Z^0_1)=\mu_1^*$ $\Pop$ a.s.. Continuing in that way as in Step 1 we can attain the upper bound $J(\mu_0)$ in the limit. In order to implement this strategy the central controller has to know $Q_n^*$ or $\mu_n^*$ at time $n$. If there is no common noise, then the sequence $(\mu_0^*,Q_0^*,\mu_1^*,Q_1^*,\ldots)$ is deterministic and we only have to know the time step $n$, so the policy is non-stationary. If the common noise is present, in order to know $Q_n^*$ the central controller has to keep track of the history $(Z_1^0,Z_2^0,\ldots)$, so the policy $\psi^N$ is history-dependent. However, we know from MDP theory that such a policy can always be dominated by a Markovian policy, so 
$$ J(\mu_0) =\lim_{N\to\infty} J^N_{\psi^N}(\mu^N_0)\le \limsup_{N\to\infty} J^N(\mu^N_0)\le  J(\mu_0)$$
which yields the statements of the theorem. \hfill$\Box$

\subsubsection{ Proof of Theorem \ref{thm:average}}

Let  $\rho=\limsup_{\beta\uparrow 1} \rho(\beta)$ and let $(\beta_n)$ be the subsequence s.t.\ $\rho=\lim_{n\to\infty} \rho(\beta_n)$. Define
$$ h(\mu) := \lim_{n\to\infty} \sup_{k\ge n} \sup_{d(\mu,\mu')\le \frac1n} h^{\beta_k}(\mu')$$
where $d$ is a metric on $\Pop(S)$. Note that $h$ is a limit of bounded, continuous functions which are decreasing in $n$ and is thus  at least upper semicontinuous.

Let us now consider the $\beta$-discounted optimality equation 
$$ J^\beta(\mu) = \tilde r(\mu,\varphi^\beta(\mu))+\beta \Eop J^\beta(\tilde T(\mu,\varphi^\beta(\mu),Z^0)) $$
where $\varphi^\beta$ is an optimal decision rule in the $\beta$-discounted model. Subtracting $\beta J^\beta(\nu)$ on both sides yields:
\begin{equation}\label{eq:beta1}
\rho(\beta) + h^\beta(\mu) = \tilde r(\mu,\varphi^\beta(\mu))+\beta \Eop h^\beta(\tilde T(\mu,\varphi^\beta(\mu),Z^0)).
\end{equation} 
From Lemma 3.4 in \cite{schal93} we know that there exist sequences $(k_n)$ of integer-valued measurable mappings and $(\mu_n)$ of $\Pop(S)$-valued measurable mappings on $\Pop(S)$ such that $k_n(\mu)\to\infty$, $\mu_n(\mu)\Rightarrow \mu$ for $n\to\infty$ and $h^{\beta_{k_n(\mu)}}(\mu_n(\mu))\to h(\mu)$. Define  $Q_n(\mu)=  f^{\beta_{k_n(\mu)}}(\mu_n(\mu))$. In what follows we fix $\mu\in \Pop(S)$ and suppress the dependence on $\mu$ in our notation. Then by \eqref{eq:beta1}
\begin{equation}\label{eq:beta2}
\rho(\beta_{k_n}) + h^{\beta_{k_n}}(\mu_n) = \tilde r(\mu_n,Q_n)+\beta \Eop h^{\beta_{k_n}}(\tilde T(\mu_n,Q_n,Z^0)).
\end{equation} 
Moreover, it follows from  \cite{schal93} Proposition 3.5 that there exists a measurable function $g^0: \Pop(S)\to \Pop(D)$ s.t.\ $g^0(\mu)$ is an accumulation point of $(Q_n(\mu))$ and $g^0(\mu)\in \tilde D(\mu)$. For the fixed $\mu$ choose a subsequence $(n_m)$ of natural numbers (for simplicity denoted by $m$) such that $Q_m(\mu) \Rightarrow g^0(\mu)$. Next note since $\tilde r$ is continuous (see proof of Theorem \ref{theo:MFLexi}) we obtain
$$ \lim_{m\to\infty} \tilde r (\mu_m,Q_m)=\tilde r (\mu, g^0(\mu)).$$ 
Since for $k_n$ large enough
$$ h^{\beta_{k_n}}(\tilde T (\mu_{k_n}, Q_{k_n}, Z^0))\le \sup_{k\ge k_n} \sup_{d(\tilde T(\mu,g^0(\mu),Z^0),\mu')\le \frac{1}{k_n}} h^{\beta_{k_n}}(\mu')$$
we obtain
$\limsup_{n\to\infty}  h^{\beta_{k_n}}(\tilde T (\mu_{k_n}, Q_{k_n}, Z^0)) \le h(\tilde T(\mu,g^0(\mu),Z^0))$. Hence taking $\limsup_{m\to\infty}$ in \eqref{eq:beta2} we obtain altogether with monotone convergence for the integral
\begin{eqnarray*}
 \rho + h(\mu) &\le&  \tilde r(\mu,g^0(\mu))+ \Eop h(\tilde T(\mu,g^0(\mu),Z^0)) \\
 &\le &  \tilde r(\mu,\varphi^*(\mu))+ \Eop h(\tilde T(\mu,\varphi^*(\mu),Z^0))
\end{eqnarray*}
 where $\varphi^*$ is a maximizer of $h$ which exists since the r.h.s.\  is upper semicontinuous and since $\tilde D(\mu)$ is compact.
 This proves part a). 
 
 Iterating this inequality $n$ times yields by (A4)
 $$ n\rho + h(\mu) \le \sum_{k=0}^{n-1} \Eop_\mu^{g^0}\left[ \tilde r (\mu_k, g^0(\mu_k))\right] + \Eop_\mu^{g^0}\left[ h(\mu_n)\right]\le  \sum_{k=0}^{n-1} \Eop_\mu^{g^0}\left[ \tilde r (\mu_k, g^0(\mu_k))\right] + L. $$
 Dividing by $n$ and taking $\liminf_{n\to\infty}$ on both sides we obtain $ \rho \le G_{g^0}(\mu)$. From Lemma \ref{lem:UB}  we deduce that $g^0$ and hence also $\varphi^*$ yield an average optimal policy. The remaining statements follow from \cite{schal93}, Proposition 3.5. \hfill$\Box$

\subsubsection{ Proof of Lemma \ref{lem:TT}}

We obtain
\begin{eqnarray}\nonumber
&&W(\mu_{k+1}^*,\mu^*) = \sup_{\|f\|_{Lip}\le 1} \Big| \int f(x) (\mu_{k+1}^*(dx)-\mu^*(dx))\Big|\\ \label{eq:w1}
&\le &   \sup_{\|f\|_{Lip}\le 1} \Big| \int\!\!\!\int\!\!\!\int \big( f(T(x,a,\mu_k^*,z))-f(T(x,a,\mu^*,z)) \big)\Pop^Z (dz) \bar Q^*(da|x) \mu_k^*(dx)\Big|\\ \label{eq:w2}
&& +  \sup_{\|f\|_{Lip}\le 1} \Big| \int\!\!\!\int\!\!\!\int  f(T(x,a,\mu^*,z)) \Pop^Z (dz) \bar Q^*(da|x) (\mu_k^*(dx)-\mu^*(dx))\Big|
\end{eqnarray}
Let us first consider the term in \eqref{eq:w1}.  By the fact that all $f$ are Lipschitz with maximal constant 1 and (T1) we obtain  that \eqref{eq:w1} can be bounded by $\gamma_W W(\mu_k^*,\mu^*)$. Thus, we consider next \eqref{eq:w2}. We show that 
$$ h(x) := \int\!\!\!\int  f(T(x,a,\mu^*,z)) \Pop^Z (dz) \bar Q^*(da|x)$$ is Lipschitz with constant bounded by $\gamma_Q\gamma_A+\gamma_S$. From this property it follows then that  \eqref{eq:w2} can be bounded by $(\gamma_Q\gamma_A+\gamma_S) W(\mu_k^*,\mu^*)$.
Hence consider
\begin{eqnarray}\nonumber
&& \Big|   \int\!\!\!\int  f(T(x,a,\mu^*,z)) \Pop^Z (dz) \bar Q^*(da|x)   -    \int\!\!\!\int  f(T(x',a,\mu^*,z)) \Pop^Z (dz) \bar Q^*(da|x')    \Big|\\ \label{eq:w3}
&\le &  \Big|   \int\!\!\!\int  f(T(x,a,\mu^*,z)) \Pop^Z (dz) (\bar Q^*(da|x)   -   \bar Q^*(da|x'))    \Big|\\\label{eq:w4}
&& + \int\!\!\!\int  \Big|   f(T(x,a,\mu^*,z))-  f(T(x',a,\mu^*,z)) \Big| \Pop^Z (dz) \bar Q^*(da|x')    
\end{eqnarray}
By (T4) we can bound \eqref{eq:w4} by $\gamma_S |x-x'|$ since $f$ is Lipschitz with constant less than 1. Now finally we have to treat \eqref{eq:w3}. Here we show that $g(a) := \int  f(T(x,a,\mu^*,z)) \Pop^Z (dz) $ is Lipschitz with constant less than $\gamma_A$:
\begin{eqnarray*}
&&  \int   \Big| f(T(x,a,\mu^*,z)) -  f(T(x,a',\mu^*,z)) \Big| \Pop^Z (dz)   \le \gamma_A |a-a'|.
\end{eqnarray*}
This altogether shows that  \eqref{eq:w2} can be bounded by $(\gamma_Q\gamma_A+\gamma_S) W(\mu_{k}^*,\mu^*)$. Finally we obtain $$ W(\mu_{k+1}^*,\mu^*)\le \gamma^{k+1} W(\mu_0^*,\mu^*)\to 0$$
for $k\to\infty$ and weak convergence follows from convergence in the Wasserstein metric. \hfill$\Box$

{\em Acknowledgement:} The author would like to thank  two anonymous  referees for their comments which helped to improve the paper.


\begin{thebibliography}{99}

\bibitem{bl} B\"auerle, N.,  Lange, D. Optimal control of partially observable piecewise deterministic Markov processes. SIAM Journal on Control and Optimization 56(2), 1441-–1462, (2018). 

\bibitem{br11} B\"auerle, N., Rieder U.
Markov Decision Processes with Applications to Finance. Springer-Verlag, Berlin Heidelberg (2011)

\bibitem{b01} B\"auerle, N.  Convex stochastic fluid programs with average cost. Journal of mathematical analysis and applications, 259(1), 137-156, (2001)

\bibitem{BT} Bertsekas, D.P.,  Tsitsiklis, J.N. Neuro-dynamic programming. Athena Scientific, Belmont, Mass. (1996).

\bibitem{bis} Biswas, A. Mean field games with ergodic cost for discrete time Markov processes. arXiv preprint arXiv:1510.08968, (2015).

\bibitem{cdf} Cao, H., J. Dianetti, Ferrari, G. Stationary discounted and ergodic mean-field games of singular control. arXiv preprint arXiv:2105.07213, (2021).

\bibitem{cd18}  Carmona, R.,  Delarue, F. Probabilistic Theory of Mean Field Games with Applications I-II. Springer Nature, (2018).

\bibitem{cltf} Carmona, R., Laurière, M., Tan, Z. Model-free mean-field reinforcement learning: Mean-field MDP and mean-field Q-learning. arXiv preprint arXiv:1910.12802, (2019).

\bibitem{clt} Carmona, R., Laurière, M.,  Tan, Z. Linear-quadratic mean-field reinforcement learning: Convergence of policy gradient methods. arXiv preprint arXiv:1910.04295, (2019).

\bibitem{CHFM} Chang, H. S., Hu, J., Fu, M. C., Marcus, S. I. Simulation-based algorithms for Markov decision processes. London: Springer, (2007).

\bibitem{eln} Elliott, R., Li, X.,  Ni, Y. H.  Discrete time mean-field stochastic linear-quadratic optimal control problems. Automatica, 49(11), 3222-3233, (2013).

\bibitem{F79} Flynn, J. Steady state policies for deterministic dynamic programs. SIAM Journal on Applied Mathematics, 37(1), 128-147, (1979).

\bibitem{gg11} Gast, N.,  Gaujal, B.  A mean field approach for optimization in discrete time. Discrete Event Dynamic Systems, 21(1), 63-101, (2011).

\bibitem{ggb12} Gast, N., Gaujal, B., Le Boudec, J. Y.  Mean field for Markov decision processes: From discrete to continuous optimization. IEEE Transactions on Automatic Control, 57(9), 2266-2280, (2012).

\bibitem{glynn} Glynn, P. W.,  Iglehart, D. L. Conditions under which a Markov chain converges to its steady state in finite time. Probability in the Engineering and Informational Sciences, 2(3), 377-382, (1988).

\bibitem{gomes} Gomes, D. A., Mohr, J.,  Souza, R. R. Discrete time, finite state space mean field games. Journal de math\'{e}matiques pures et appliqu\'{e}es, 93(3), 308-328, (2010).

\bibitem{ggw19} Gu, H., Guo, X., Wei, X.,  Xu, R.  Dynamic programming principles for learning MFCs. arXiv preprint arXiv:1911.07314, (2019).

\bibitem{ggw20}  Gu, H., Guo, X., Wei, X.,  Xu, R. Q-Learning for Mean-Field Controls. arXiv preprint arXiv:2002.04131, (2020).

\bibitem{HLL} Hern\'{a}ndez-Lerma, O.,  Lasserre, J. B. Average optimality in Markov control processes via discounted-cost problems and linear programming. SIAM journal on control and optimization, 34(1), 295-310, (1996).

\bibitem{hjm} Higuera-Chan, C. G., Jasso-Fuentes, H.,  Minj\'{a}rez-Sosa, J. A. Discrete-time control for systems of interacting objects with unknown random disturbance distributions: A mean field approach. Applied Mathematics {\&} Optimization, 74(1), 197-227, (2016).

\bibitem{hjm2}  Higuera-Chan, C. G., Jasso-Fuentes, H.,  Minj\'{a}rez-Sosa, J. A. Control systems of interacting objects modeled as a game against nature under a mean field approach. Journal of Dynamics \& Games, 4(1), 59., (2017)

\bibitem{HY} Hordijk, A.,  Yushkevich, A. A. Blackwell optimality in the class of stationary policies in Markov decision chains with a Borel state space and unbounded rewards. Mathematical methods of operations research, 49(1), 1-39, (1999).

\bibitem{JR88} Jovanovic, B., Rosenthal, R.W. Anonymous sequential games. J. Math. Econ 17, 77–87, (1988)

\bibitem{ll}  Lasry, J. M. Lions, P. L. Jeux \`{a} champ moyen. I – Le cas stationnaire. Comptes Rendus Math\'ematique (in French). 343 (9): 619–625, (2006)

\bibitem{LYCR} Li, S. H., Yu, Y., Calderone, D., Ratliff, L., A\c{c}ikme\c{s}e, B. Tolling for constraint satisfaction in Markov decision process congestion games. In 2019 American Control Conference (ACC) (pp. 1238-1243). IEEE, (2019).

\bibitem{Mc} McKean, H. P. A class of Markov processes associated with nonlinear parabolic equations. Proc. Natl. Acad. Sci. USA. 56 (6): 1907–1911, (1966).

\bibitem{mp} Motte, M.,  Pham, H. Mean-field Markov decision processes with common noise and open-loop controls. arXiv preprint arXiv:1912.07883, To appear in Annals of Applied Probability (2021+).

\bibitem{ps} Peyrard, N., Sabbadin, R.  Mean field approximation of the policy iteration algorithm for graph-based Markov decision processes. Frontiers in Artificial Intelligence and Applications, 141, 595, (2016).

\bibitem{pw} Pham, H.,  Wei, X.  Discrete time McKean–Vlasov control problem: A dynamic programming approach. Applied Mathematics \& Optimization, 74(3), 487-506, (2016).

\bibitem{P07} Powell, W. B.  Approximate Dynamic Programming: Solving the curses of dimensionality (Vol. 703). John Wiley \& Sons, (2007)

\bibitem{rs} Rudolf, D.,  Schweizer, N. Perturbation theory for Markov chains via Wasserstein distance. Bernoulli, 24(4A), 2610-2639, (2018).

\bibitem{saldi} Saldi, N., Basar, T.,  Raginsky, M. Markov--Nash equilibria in mean-field games with discounted cost. SIAM Journal on Control and Optimization, 56(6), 4256-4287, (2018).

\bibitem{schal93} Sch\"al, M.  Average optimality in dynamic programming with general state space. Mathematics of Operations Research, 18(1), 163-172, (1993).

\bibitem{sen09} Sennott, L. I. . Stochastic dynamic programming and the control of queueing systems (Vol. 504). John Wiley \& Sons, (2009).


\bibitem{wein05} Weintraub, G. Y., Benkard, L., Van Roy, B.  Oblivious equilibrium: A mean field approximation for large-scale dynamic games. Advances in neural information processing systems, 18, 1489-1496, (2005).

\bibitem{wi20} Wi\c{e}cek, P.  Discrete-time ergodic mean-field games with average reward on compact spaces. Dynamic Games and Applications, 10(1), 222-256, (2020).
\end{thebibliography}
\end{document}